\documentclass[a4paper,12pt]{amsart}
\usepackage{amssymb,amsthm,amsmath,amsfonts}
\usepackage{hyperref}
\usepackage[margin=1.9cm,a4paper]{geometry}

\theoremstyle{definition}

\newtheorem{Cor}{Corollary}

\newtheorem{Prop}{Proposition}
\newtheorem{Defn}{Definition}

\newtheorem{definition}{Definition}
\newtheorem{example}{Example}
\newtheorem{lemma}{Lemma}
\newtheorem{remark}{Remark}
\newtheorem{theorem}{Theorem}

\newcommand{\norm}[1]{\left\Vert#1\right\Vert}

\newcommand{\primes}{'}

\newcommand{\set}[1]{\left\{#1\right\}}

\usepackage{mathtools}

\DeclareMathOperator{\Part}{\mathcal{P}}

\DeclareMathOperator{\NC}{\mathit{NC}}

\newcommand{\mf}[1]{\mathbb{#1}}
\newcommand{\mc}[1]{\mathcal{#1}}
\newcommand{\mb}[1]{\mathbf{#1}}

\newcommand{\id}{\mathrm{Id}}
\newcommand{\ip}[2]{\left \langle #1, #2 \right \rangle}

\usepackage{cancel}

\newcommand{\rc}{\mathrm{rc} }

\def\state{\varphi }
\def\A{a_{q,t}}
\def\pa{\tilde{p}}
\def\C{{\mathbb C}}
\def\c{{ f}}
\def\R{{\mathbb R}}
\def\N{{\mathbb N}}
\def\Z{{\mathbb Z}}
\def\Cum {{\mathrm K}^\mb{x}}
\def\Cumm {{\mathrm R}^\mb{x}_{\pi_f}}
\def\CummTensor {\widehat{\mathrm{R}}^\mb{x}_{\pi_f}}

\def\6{\, {\rm d}}

\def\B{b_{\alpha,q}}

\def\X{ \mathcal{B}_{\alpha,q}^{\lambda}}
\def\Y{ \mathrm{Y}_{q,t}}

\def\r{r}
\def\F{\mathcal{F}_{\rm fin}(H)}
\def\FD{\mathcal{F}_{\rm fin}(\mc{D})}
\def\id{I}

\def\P{\mathcal{P}}
\def\NC{\mathcal{NC}}

\def\PB{\mathcal{P}^B}

\def\InS{\text{\normalfont MaxC}}
\def\OptT{\mb{T}^{\mb{x}}}
\def\OptTTyl{\mb{\widehat{T}}^{\mb{x}}}
\def\InNB{\text{\normalfont rnarc}}
\def\InNA{\text{\normalfont rarc}}

\def\NB{\text{\normalfont Narc}}
\def\SLNB{\text{\normalfont MaxL}}
\def\SLNBlitvic{\text{\normalfont MLeft}}

\def\Sing{\text{\normalfont Sing}}

\def\nest{\text{\normalfont nest}}

\newcommand{\Semi}{{\normalfont Arc} }
\newcommand{\BC}{\mb{{B}}_\textbf{c} }
\newcommand{\ith}{\textbf{ i} }
\usepackage{amsmath}

\usepackage{pst-node}
\usepackage{enumerate}
\usepackage[mathscr]{eucal}

\numberwithin{equation}{section}

\usepackage{times}

\def\cvput#1[#2]{\pnode(#1,1){#1} \pscircle*(#1,1){.1} \rput(#1,.5){$#2$}}

\def\cvpuW#1[#2]{\pnode(#1,1){#1} \pscircle*(#1,1){.1} }

\def\cvpuDots#1[#2]{\pnode(#1,1){#1} \dots }

\newcommand{\Comment}[1]{}

\newcommand{\Tens}{\OptTTyl}

\title[Poisson type operators on the Fock space of type B and in the Blitvi{\'c} model]
{Poisson type operators on the Fock space of type B and in the Blitvi{\'c} model
}
\author[Wiktor Ejsmont]{Wiktor Ejsmont}
\address {
Mathematical Institute, University of Wroc\l aw \\
pl. Grunwaldzki 2/4, 50-384 Wroc\l aw, Poland}
\email{wiktor.ejsmont@gmail.com}

\begin{document}
\maketitle
\begin{abstract}
In \cite{Bia97} Biane proposed a new statistic on set partitions which
he called \emph{restricted crossings}.
In a series of papers \cite{Ans01,Ans04,Ans04b,Ans05} 
Anshelevich showed that this statistic is  an essential tool to 
investigate stochastic processes on $q$-Fock space.
In particular, Anshelevich constructed operators whose
moments count restricted crossings and used these operators
to develop a beautiful
  theory of noncommutative $q$-L\'{e}vy processes.  
In the present paper following Anshelevich we define gauge operators on $(\alpha,q)$-Fock and cumulants which are governed by statistics on  partitions of type B.
In addition 
 we investigate this construction in the context of a model of Blitvi{\'c} model \cite{B12}, where some related but different combinatorial structures appear,
 and we explain their relation with $t$-free probability.
\end{abstract}
\section{Introduction}

It is a basic fact in quantum field theory
that  field operators, that is sums of creation and annihilation operators on 
symmetric Fock space give rise to Gaussian distributions in the vacuum state.
Hudson and Parthasarathy \cite{HudPar} and Sch{\"u}rmann \cite{SchurCondPos} observed that by adding an appropriate gauge component
one obtains an operator with Poisson distribution and more generally,
arbitrary infinitely divisible distributions (provided all moments are finite) 
can be modeled on symmetric Fock space in this way.
Such distributions are characterized
by conditionally positive definite sequences of cumulants.
In \cite{BS91} Bo\.zejko and Speicher
constructed $q$-Fock spaces
by deforming the inner product on full Fock space
with the aid of positive definite functions on the symmetric group
depending on a continuous parameter $q$.
The corresponding field operators give rise to
$q$-Gaussian distributions, i.e.,
the orthogonalizing measures for the classical Rogers $q$-Hermite polynomials.
Anshelevich \cite{Ans01} constructed the corresponding gauge operators
and showed that they give rise to $q$-Poisson and other infinitely
divisible distributions, still being indexed by conditionally positive
definite sequences of so-called $q$-cumulants.
The main new combinatorial ingredient in this construction was the notion
of \emph{restricted crossings} defined earlier by Biane \cite{Bia97}.
It is worth noting that for $q = 0$ all these results restrict to free
probability see \cite{V85,V1,NS06}, thus providing operator
representations of freely infinitely divisible distributions.

The paper \cite{BEH15} exhibited a new kind of generalized Gaussian
processes from positive definite functions on Coxeter groups of type
$B$; that is, the symmetric group, which is a Coxeter group of type
$A$, is replaced by a Coxeter group of type $B$.
More precisely, in \cite{BEH15} we constructed a positive definite
function depending on two parameters $(\alpha,q)$, the corresponding
Fock space and Gaussian operators acting on them, and identified a new
combinatorial statistic underlying their distribution.
In the present article we first define a gauge operator on
$(\alpha,q)$-Fock space in the spirit of Anshelevich's approach.
In the calculation of the joint moments of these operators
new partitions, statistics and cumulants of type $B$ arise.
It turns out that we have to consider partitions with signed or bicolored
blocks, similar to the appearance of signed permutations
(i.e., permutations of the numbers $\pm 1, \dots ,\pm n$ 
arising in the natural definition of Coxeter groups of type $B$). 

The relation between Theorem~3 of the present article and  the main
result of \cite[Theorem 3.7]{BEH15}  is the following.
We recover the $(\alpha,q)$-formula for Gaussian moments
\cite[Equation (3.14)]{BEH15} by setting the gauge operator $T=0$ in Theorem 3.
Then only pair partitions contribute, there are no extended blocks
and the corresponding partition statistic coincides with the function from
 \cite{BEH15}.
%
Here the restricted crossings ($\rc$) of a partition are
those of Biane \cite{Bia97}.
At the beginning we give an informal definition of restricted
crossings as the number of intersections of arcs above the $x-$axis in
the corresponding diagram (see Figure
\ref{RestrictedCrossingsIntro2}).
  Moreover, there appears
another new statistic, which we called \emph{restricted negative nesting}
($\InNB$) and which counts the number of negative arcs covered
by other arcs (see Figure \ref{RestrictedCrossingsIntro2}). Note that
in this article the arcs shown in the figures with color $-1$ and $1$ are
denoted by - - - and ---, respectively.

\begin{figure}[h]

\begin{center}
\includegraphics[width=0.7\textwidth]{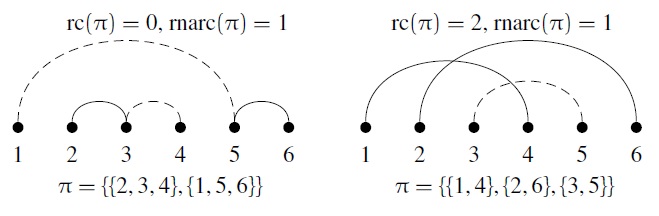}
\caption{ Examples of restricted crossing partition of type B of $6$ elements. 
}
 \label{RestrictedCrossingsIntro2}
\end{center}
\end{figure}

Let us now turn to the construction of $q$-deformed Fock space of
Bo\.zejko and Spei\-cher \cite{BS91}. 
Let  $-1<q<1$ and denote by
 $S(n)$ the symmetric group on $\{1,\dots,n\}$.
For a permutation $\sigma\in S(n)$ denote by
$|\sigma|:=\#\{(i,j):i<j,\sigma(i)>\sigma(j)\}$  the number of inversions.
For a separable real Hilbert space $H_\R$ with complexification $H$ 
we can equip the direct sum
$\mathcal{F}_q(H)=(\mathbb{C}\Omega)\oplus\bigoplus_{n=1}^\infty
H^{\otimes n}, $ where $\Omega$ denotes the vacuum vector,
with the $q$-deformed inner product induced by the
sesquilinear extension of
\begin{align*} \label{q-inner}
&\left\langle x_1\otimes\dots\otimes  x_m, y_1\otimes\dots\otimes
y_n\right\rangle_q = \delta_{m,n}\sum_{\sigma\in S(n)}q^{|\sigma|}\prod_{j=1}^n \langle x_j,y_{\sigma(j)}\rangle.
\end{align*}
On this space the creation and annihilation operators
\begin{align*}
& a_q^\ast(x) x_1 \otimes x_2 \otimes \ldots \otimes x_n=x_1 \otimes x_2 \otimes \ldots \otimes x_n \otimes x, \\
& a_q(x) x_1 \otimes x_2 \otimes \ldots \otimes x_n = \sum_{i=1}^n q^{n-i} \langle x, x_i\rangle \otimes x_1 \otimes \ldots \otimes \check{x}_i \otimes \ldots \otimes x_n
\end{align*}
are adjoint to each other; here superscript $\check{x}_i$ indicates
that ${x}_i $ is deleted from the product.
Note that in order to be compatible with the article \cite{BEH15}
all creation and annihilation operators  in this paper act from the right.
Then the right version of Anshelevich's construction \cite{Ans01} goes as follows.
Let $T$ be a bounded self-adjoint operator.
 The corresponding {gauge operator} $p(T)$ is the operator on
$\mathcal{F}_q(H)$ defined by
\begin{align*}
& p(T) \Omega = 0, 
\\
& p(T) x_1 \otimes x_2 \otimes \ldots \otimes x_n = \sum_{i=1}^n q^{n-i} x_1 \otimes \ldots \otimes \check{x}_i \otimes \ldots \otimes x_n \otimes (T x_i),
\end{align*}
for $x_i\in H$.
 Then the joint moments of a process $A_q(x)=a_q(x)+a_q(x)+p(T_x)$,
 indexed by $x\in H_\R$,
 with respect to vacuum expectation on  $q$-Fock space are given by
\begin{align*}
\langle\Omega, A_q(x_n)\cdots A_q(x_1)\Omega\rangle_{q}= \sum_{\pi\in \P_{\geq 2}(n)} q^{\rc(\pi)} \prod_{\substack{B \in \pi} }\Big\langle x_{\max(B)} ,\prod_{\substack{ i\in B\\ i\neq \min(B) ,\max(B)}} T_{x_{i}} x_{\min(B)} \Big\rangle, 
\end{align*}
where by $\P_{\geq 2}(n)$ we denote the set partitions of
$\{1,\dots,n\}$ without singletons and $\rc(\pi)$ denotes the number
of restricted crossings of $\pi$ in the sense of Biane \cite{Bia97}.

The paper is organized as follows. First we present the definition of $(\alpha,q)$-Fock space and the creation and annihilation operators acting on it.
In subsection 2.2 we introduce gauge operators and some of their natural properties, including norm estimates and the self-adjointness.
In subsections 2.3 and 2.4 we present partitions of type $B$ and
the relevant statistics.
The generalized processes and cumulants of type $B$ are studied in
Section~3, the main result being an explicit Wick formula for the
mixed moments.
Finally, in Section~4 we show that some similar combinatorial objects
appear in the context of $(q,t)$-probability spaces introduced by
Blitvi{\'c} \cite{B12}.

\section{Preliminaries  }
\subsection{$(\alpha,q)$-Fock space and creation, annihilation operators of type B} 
Let $\Sigma(n)$ be the set of bijections $\sigma$ of the $2n$ points $\{\pm1,\cdots, \pm n\}$ such that 
$\sigma(-k)=-\sigma(k), k=1,\dots,n$.  Equipped with the composition operation as a product, $\Sigma(n)$ becomes a group and is called a \emph{Coxeter group of type B} or a \emph{hyperoctahedral group}. 
The Coxeter group $\Sigma(n)$ is generated by $\pi_0=(1,-1), \pi_i =(i,i+1)$, $i=1,\dots,n-1$ and will be denoted by $\Sigma(n)$. These generators satisfy the generalized braid relations $\pi_i^2=e, 0\leq i \leq n-1$, $(\pi_0\pi_1)^4=(\pi_i \pi_{i+1})^3=e, 1\leq i < n-1$ and $(\pi_i \pi_j)^2=e$ if $|i-j|\geq2, 0\leq i,j\leq n-1$. Note that $\{\pi_i\mid i=1,\dots,n-1\}$ generate a symmetric group $S(n)$. 

Let $H_\R$ be a separable real Hilbert space and let $H$ be its complexification with inner product $\langle\cdot,\cdot\rangle$, linear on the right component and anti-linear on the left. When corresponding elements are in $H_\R$, it holds true that $\langle x,y\rangle=\langle y,x\rangle$.
We also assume that there exists a  selfadjoint involution $H\ni x\mapsto \bar{x}\in H$, which is a selfadjoint linear bounded operator on $H$ such that the double application of it becomes the identity operator.
   Let $\F$ be the (algebraic) full Fock space over $H$
\begin{equation}
\F:= \bigoplus_{n=0}^\infty H^{\otimes n} ,\label{fockspace}
\end{equation} 
with convention that $H^{\otimes 0}=\C\Omega$ is a one-dimensional normed unit vector. Note that elements of $\F$ are finite linear combinations of the elements from $H^{\otimes n}, n\in \N\cup\{0\}$ and we do not take the completion. For $\alpha, q \in(-1,1)$ we define the type B symmetrization operator on $H^{\otimes n}$, 
\begin{align}
&P_{\alpha,q}^{(n)}= \sum_{\sigma \in \Sigma(n)} \alpha^{l_1(\sigma)} q^{l_2(\sigma)}\, \sigma,\qquad n \geq1, \\ 
&P_{\alpha,q}^{(0)}= \id_{H^{\otimes 0}}, 
\end{align}
where $\sigma \in \Sigma(n)$ is in an irreducible form with minimal length and in this case let 
\begin{align*}
&l_1(\sigma) \text{ be the number of $\pi_0$ appearing in $\sigma$}, \\
&l_2(\sigma) \text{ be the number of $\pi_i, 1 \leq i \leq n-1$, appearing in $\sigma$},
\\&\pi_i(x_1\otimes\cdots \otimes x_n) = x_1 \otimes \cdots \otimes x_{i-1} \otimes x_{i+1} \otimes x_{i} \otimes x_{i+2}\otimes \cdots \otimes x_{n},& n \geq2,  \\
&\pi_0(x_1\otimes\cdots \otimes x_n)= \overline{x_1}\otimes x_2\otimes\cdots \otimes x_n,& n \geq 1. 
\end{align*}
Moreover let 
$$
P_{\alpha,q}:=\bigoplus_{n=0}^\infty P_{\alpha,q}^{(n)},
$$ be the type-B symmetrization operator acting on the algebraic full Fock space. We deform the inner product by using the type-B symmetrization operator:
\begin{equation}
\langle x_1 \otimes \cdots \otimes x_m, y_1 \otimes \cdots \otimes y_n\rangle_{\alpha,q}:=\langle x_1 \otimes \cdots \otimes x_m, P_{\alpha,q}^{(n)}(y_1 \otimes \cdots \otimes y_n)\rangle_{0,0}, 
\end{equation}
where $\langle x_1 \otimes \cdots \otimes x_m, y_1 \otimes \cdots \otimes y_n\rangle_{0,0}:= \delta_{m,n}\prod_{i=1}^n \langle x_i, y_i\rangle $. The $(\alpha,q)$-inner product is a semi-inner product for $\alpha,q\in[-1,1]$. We restrict the parameters to the case $\alpha,q \in(-1,1)$ so that the deformed semi-inner product is an inner product. 
\begin{Defn} For $\alpha,q\in(-1,1)$, the algebraic full Fock space $\F$ equipped with the inner product $\langle\cdot,\cdot \rangle_{\alpha,q}$ is called the \emph{Fock space of type B} or the \emph{$(\alpha,q)$-Fock space}. Let 
$$\B(x):=\r_q(x)+ \alpha \ell^N_{q}(x), \qquad x \in H, 
$$ 
where 
\begin{align}
\r_q(x)(x_1\otimes \cdots \otimes x_n)&= \sum_{k=1}^n q^{n-k} \langle x, x_k \rangle\, x_1\otimes \cdots \otimes \check{x}_k \otimes \cdots \otimes x_n, \qquad \r_q(x)\Omega=0,\label{rq}\\
\ell^N_q(x)(x_1\otimes \cdots \otimes x_n)&= q^{n-1}\sum_{k=1}^n q^{k-1} \langle x,\bar{x}_k\rangle\, x_1\otimes \cdots \otimes \check{x}_k \otimes \cdots \otimes x_n, \qquad \ell^N_q(x)\Omega=0. \label{lq}
\end{align}
Operator $\B^\ast(x)$ is its adjoint with respect to the inner product $\langle\cdot,\cdot \rangle_{\alpha,q}$. The operators $\B^\ast(x)$ and $\B(x)$ are called \emph{creation and annihilation operators of type B} or \emph{$(\alpha,q)$-creation and annihilation operators}. Denote by $\state$ the vacuum vector state $\state(X)=\langle\Omega, X\Omega\rangle_{\alpha,q}$ and we denote $\state_{\alpha,q}(X)=\state(X)$.
\end{Defn}
\begin{remark} (1). For $x\in H$ the creation operator of type B has the following form
\begin{align*}
&\B^\ast(x)(x_1 \otimes \cdots \otimes x_n):= x_1 \otimes \cdots \otimes x_n \otimes x, &&n\geq 1\\
&\B^\ast(x)\Omega=x.
\end{align*}
(2). Operators $\B^\ast(x),\B(x)$ are bounded, for $\alpha,q\in(-1,1)$; see \cite{BEH15}. 
\end{remark}
\noindent
\textbf{Additional facts.}
\noindent In this part let us recall the properties of creation and annihilation operators from \cite{BEH15}. We will use them in the next subsection.
\begin{Prop}\label{prop1} We have the decomposition 
\begin{equation}\label{decomposition}
P^{(n)}_{\alpha,q}=(P^{(n-1)}_{\alpha,q}\otimes I)R^{(n)}_{\alpha,q}={R^{(n)}_{\alpha,q}}^*(P^{(n-1)}_{\alpha,q}\otimes I) \text{~on $H^{\otimes n}$}, \qquad n\geq 1, 
\end{equation}
where
\begin{equation}
R^{(n)}_{\alpha,q} = 1+\sum_{k=1}^{n-1}q^{k}\pi_{n-1}\cdots \pi_{n-k} + \alpha q^{n-1}\pi_{n-1} \pi_{n-2} \cdots \pi_{1}\pi_0\left(1+\sum_{k=1}^{n-1}q^{k}\pi_{1}\cdots \pi_{k}\right), 
\end{equation}
and the adjoint ${R^{(n)}_{\alpha,q}}^*$ is taken with respect to $\langle\cdot,\cdot \rangle_{0,0}$. 
\end{Prop}

\begin{theorem}\label{thm1}For $n\geq 1$, we have
$$
\B(x)= \r(x) R^{(n)}_{\alpha,q},
\qquad x \in H,
$$
where $\r(x)$ is the free right annihilation operator.
\end{theorem}
\begin{lemma} \label{lem4} For $x \in H$, we get 
\begin{align}
&\|R_{\alpha,q}^{(n)}\|_{0,0} \leq (1+|\alpha| |q|^{n-1})[n]_q,\qquad n \geq 1. 
\end{align}
\end{lemma}
\subsection{Gauge operators}In this subsection we define a differential second quantization operator. First we introduce an operator which acts on $(\alpha,q)$-Fock space, as
\begin{align*}
& p_0(T) \Omega = 0, \\
& p_0(T) x_1 \otimes x_2 \otimes \ldots \otimes x_{n-1}\otimes x_n = x_1 \otimes x_2 \otimes \ldots \otimes \ldots \otimes x_{n-1}\otimes (T x_n),
\end{align*}
where $T$ is an operator on $H$ with dense domain $\mc{D}$, where we assume that our involution on $H$ and domain $\mc{D}$ are related by $\overline{\mc{D}}=\mc{D}$. The adjoint of this operator satisfies $\langle p_0(T)f | \zeta \rangle_{0,0}=\langle f | p_0(T^*)\zeta \rangle_{0,0}$, and allows us to define a gauge operator (preservation or differential second quantization). 

\begin{Defn}The \emph{gauge operator} $p(T)$ is an operator on $\F$ with dense domain $\FD$ defined by 
\begin{align*}
& p(T) \Omega = 0, \\
& p(T)  = p_0(T) R^{(n)}_{\alpha,q}.
\end{align*}
\end{Defn}

\noindent We also introduce the notation 
\begin{align}
& \r_q^{T} :=p_0(T)\Bigg(1+\sum_{k=1}^{n-1}q^{k}\pi_{n-1}\cdots \pi_{n-k}\Bigg) \label{rqT}, \\
& \ell_{q}^{N,T}:= p_0(T)q^{n-1}\Bigg(\pi_{n-1} \pi_{n-2} \cdots \pi_{1}\pi_0\left[1+\sum_{k=1}^{n-1}q^{k}\pi_{1}\cdots \pi_{k}\right]\Bigg), \label{lqT}
\end{align}
i.e. $p(T)=\r_q^{T} +\alpha \ell_{q}^{N,T}$. Sometimes we will use the notation $p_{\alpha,q}(T)=p(T)$.
\begin{remark}(1). Let us observe that directly from the generalized braid relations for $k\in [n-1]$, we have
\begin{align*} 
 & \pi_{n-1}\cdots \pi_{n-k}=\bigl(\begin{smallmatrix}
  1 & \cdots & n-k & \cdots & n-1 & n \\
  1 & \cdots & n-k+1 & \cdots & n & n-k
 \end{smallmatrix}\bigr), \textrm{ }
 \pi_{n-1} \pi_{n-2} \cdots \pi_{1}\pi_0\pi_{1}\cdots \pi_{k}=\bigl(\begin{smallmatrix}
   1 & \cdots & k+1 & \cdots & n-1 & n \\
  1 & \cdots  & k+2 & \cdots & n & -(k+1)
 \end{smallmatrix}\bigr)
  \\& \textrm{and }\pi_{n-1} \pi_{n-2} \cdots \pi_{1}\pi_0=\bigl(\begin{smallmatrix}
   1 & 2 &\cdots & n-1 & n \\
  2 & 3 &\cdots & n & -1
 \end{smallmatrix}\bigr).
 \end{align*} 
This observation allows us to rewrite action of $\r_q^{T}$ and $\ell_q^{N,T}$ on $H^{\otimes n}$ as
\begin{align} 
\r_q^{T} (x_1\otimes \cdots \otimes x_n)&= \sum_{k=1}^n q^{n-k}  \, x_1\otimes \cdots \otimes \check{x}_k \otimes \cdots \otimes x_n\otimes T(x_k), \\
\ell_q^{N,T}(x_1\otimes \cdots \otimes x_n)&= q^{n-1}\sum_{k=1}^n q^{k-1} \, x_1\otimes \cdots \otimes \check{x}_k \otimes \cdots \otimes x_n\otimes T(\bar{x}_k). 
\end{align}
(2). The operator $p_{0,q}(T)$ is precisely the same as in \cite{Ans01} (see introduction of this article). Visually, $p_{0,q}(T)$ looks different than in \cite{Ans01}, but they are identical (in the sense that they are isomorphic) because we now use right creators. 
We would also like to emphasize that in the literature one can find other operators of this type but in general they are not symmetric; see \cite{Mol,Sni}. 
\end{remark}
\begin{Prop} \label{Prop:samosprzezone}If $T$ is essentially self-adjoint on a dense domain $\mc{D}$ and $T(\mc{D}) \subset \mc{D}$, then $p(T)$ is essentially self-adjoint on the dense domain $\FD$.
\end{Prop}

\begin{proof}We first observe that $p_0(T^*) (P^{(n-1)}_{\alpha,q}\otimes I)= (P^{(n-1)}_{\alpha,q}\otimes I)p_0(T^*)$, indeed for $x_1 \otimes x_2 \otimes \ldots \otimes x_n \in \mc{D}^{\otimes n}$, we have
\begin{align*}&p_0(T^*) (P^{(n-1)}_{\alpha,q}\otimes I)(x_1 \otimes \ldots \otimes x_n)=P^{(n-1)}_{\alpha,q}(x_1 \otimes \ldots \otimes x_{n-1})\otimes T^*(x_n)\\&= (P^{(n-1)}_{\alpha,q}\otimes I)(x_1 \otimes \ldots \otimes x_{n-1}\otimes T^*(x_n))=(P^{(n-1)}_{\alpha,q}\otimes I)p_0(T^*) (x_1 \otimes \ldots \otimes x_n).
\end{align*}
Now we show that $p(T)$ is symmetric on $\FD$. Let us fix $n$, and $f,g \in \mc{D}^{\otimes n}$, then
\begin{align*}
\ip{p(T) f}{g}_{\alpha,q} &=\langle p(T) f, P^{(n)}_{\alpha,q} g \rangle_{0,0} =\langle p_0(T) R^{(n)}_{\alpha,q} f,  (P^{(n-1)}_{\alpha,q}\otimes I) R^{(n)}_{\alpha,q}g \rangle_{0,0} \\
&= \langle R^{(n)}_{\alpha,q} f, p_0(T^*) (P^{(n-1)}_{\alpha,q}\otimes I) R^{(n)}_{\alpha,q}g \rangle_{0,0}
 = \langle R^{(n)}_{\alpha,q} f,  (P^{(n-1)}_{\alpha,q}\otimes I)p_0(T^*) R^{(n)}_{\alpha,q}g \rangle_{0,0}
\intertext{by Proposition \ref{prop1}, we have}
&=\langle  f,  {R^{(n)}_{\alpha,q}}^*(P^{(n-1)}_{\alpha,q}\otimes I)p_0(T^*) R^{(n)}_{\alpha,q}g \rangle_{0,0}
= \ip{f}{p(T^*)g}_{ \alpha,q}.
\end{align*}

\noindent
Now we show that $T$ is essentially self-adjoint. Let $E$ be the spectral measure of the closure of $T$ and $C \in \mf{R}_+$. Let $\set{x_i}_{i=1}^n \subset (E_{[-C. C]} H) \cap \mc{D}$, $\vec{x} = x_1 \otimes x_2 \otimes \ldots \otimes x_n$, then $\norm{T x_i} \leq C \norm{x_i}$ and
\begin{align*}
&\ip{p(T) \vec{x}}{ p(T) \vec{x}}_{0,0}=\ip{p_0(T){R^{(n)}_{\alpha,q}} \vec{x}}{ p_0(T){R^{(n)}_{\alpha,q}} \vec{x}}_{0,0}\leq C^2 \norm{{R^{(n)}_{\alpha,q}}\vec{x}}_{0,0}^2 
\intertext{by Lemma \ref{lem4}, we have}
&\leq \big(C(1+|\alpha| |q|^{n-1})[n]_q \norm{\vec{x}}_{0,0}\big)^2 
\leq \big(2nC \norm{\vec{x}}_{0,0}\big)^2.
\end{align*}
Thus we get the following estimation for the norm of $p(T)^k $
\begin{align*}
&\norm{p(T)^k \vec{x}}_{ \alpha,q}^2 = \ip{p(T)^k \vec{x}}{P_{ \alpha,q}^{(n)} p(T)^k \vec{x}}_{0,0} 
\leq \norm{P_{ \alpha,q}^{(n)}}_{0,0}^2\ip{p(T)^k \vec{x}}{ p(T)^k \vec{x}}_{0,0}\\&
\leq \norm{P_{ \alpha,q}^{(n)}}_{0,0}^2 (2^kn^k C^k \norm{\vec{x}}_{0,0})^2.
\end{align*}
Now, we use the estimations from the proof of \cite[Theorem 2.9, equation (2.32)]{BEH15} i.e. 
\begin{equation}
\begin{split}
P_{\alpha,q}^{(n)} \leq (1+|\alpha\|q|^{n-1})[n]_q (P_{\alpha,q}^{(n-1)}\otimes I) ,
\end{split}
\end{equation}
with respect to the $(0,0)$-inner product, so $\norm{P_{ \alpha,q}^{(n)}}_{0,0} \leq \prod_{i=1}^n (1+|\alpha||q|^{i-1})[i]_{q} \leq 2^n n!$. It can be shown that $\norm{p(T)^k \vec{x}}_{\alpha,q} \leq {2^n n!}2^k n^k C^k \norm{\vec{x}}_{0,0}$, so the  series $\sum_{k=0}^\infty\frac{p(T)^k \vec{x}}{k!}s^k$ has a positive radius of absolute convergence, because 
\[
\limsup_{k \rightarrow \infty}\sqrt[k]{ \frac{\norm{p(T)^k \vec{x}}_{ \alpha,q}}{k!}}\leq \limsup_{k \rightarrow \infty}\sqrt[k]{ \frac{{2^n n!} 2^kn^k C^k \norm{\vec{x}}_{0,0}}{k!}} = 0.
\]
Therefore $\vec{x}$ is an analytic vector for $p(T)$. The linear span of such vectors is invariant under $p(T)$ and is a dense subset of $\mc{D}^{\otimes n}$. Therefore by Nelson's analytic vector theorem \cite{Nel59} (see also \cite{ReeSim1}), $p(T)$ is essentially self-adjoint on $\mc{D}^{\otimes n}$.
\end{proof}

\begin{Prop} If $T$ is a bounded operator on $H$, then $p(T)$ is a bounded operator on the $(\alpha,q)$-Fock space.
\label{Ansh+}
\end{Prop}
\begin{proof} We begin by showing that $p(T)$ is bounded in $\mathcal{F}_{0,0}(H)$. Next we show that $p(T)$ is bounded. Using Proposition \ref{Prop:samosprzezone} it can be shown that $P_{\alpha,q}p(T^{*})=p(T)^{*}P_{\alpha,q}$ where $p(T)^{*}$ is taken with respect to the $(0,0)$-inner product. Indeed, for $f,g \in H^{\otimes n}$, we have  \begin{align*}
&\langle f, p(T^*) g \rangle_{\alpha,q}= \ip{ f}{P^{(n)}_{\alpha,q} p(T^*)g}_{0,0}
\intertext{ by Proposition \ref{Prop:samosprzezone}, we get} 
&=\ip{p(T) f}{g}_{\alpha,q} =\langle p(T) f, P^{(n)}_{\alpha,q} g \rangle_{0,0}= \langle f, p(T)^*P^{(n)}_{\alpha,q} g \rangle_{0,0}.
\end{align*}
This gives us $P_{\alpha,q}p(T^{*})p(T)=p(T)^{*}P_{\alpha,q}p(T)\geqslant 0 $ and 
$$
P_{\alpha,q}p(T^{*})p(T) [p(T^{*})p(T)]^*P_{\alpha,q}\leqslant \|p(T^{*})p(T) [p(T^{*})p(T)]^*\|_{0,0}P_{\alpha,q}^2.
$$
By taking the square root of the operators from above inequality, we get
\begin{align*} \nonumber
P_{\alpha,q}p(T^{*})p(T)\leqslant \sqrt{\|p(T^{*})p(T) [p(T^{*})p(T)]^*\|_{0,0}}P_{\alpha,q} \leqslant \|p(T^{*})\|_{0,0}\|p(T)\|_{0,0}P_{\alpha,q}. \label{eq:nierownosc1}
\end{align*}
If we take $f\in H^{\otimes n}$, then we get
\begin{align*} \nonumber
&\langle p(T)f | p(T)f \rangle_{\alpha,q} = \langle f | p(T^*)p(T)f \rangle_{\alpha,q} = \langle f | P_{\alpha,q}p(T^*)p(T)f \rangle_{0,0}. 
&
\end{align*}
It is clear by the definition of $p_{0}$ that $\|p_{0}\|_{0,0}\leqslant \|T\|$, and thus 
\begin{align*} &\|p(T)f\|_{0,0}=\|p_0(T)R^{(n)}_{\alpha,q} f\|_{0,0}\leqslant \|p_0\|_{0,0}\| R^{(n)}_{\alpha,q}f\|_{0,0}.
\intertext{Now we use the estimation from Lemma \ref{lem4} and we get}
& \leqslant \|T\|(1+|\alpha| |q|^{n-1})[n]_q\|f\|_{0,0} \leqslant \max\{1+|\alpha| ,(1+|\alpha| )/(1-q)\}\|T\|\|f\|_{0,0}.
\end{align*}
Finally, since $\|T^*\|=\|T\|$, we conclude that
$$
\| p(T) \|_{\alpha,q} \leqslant \sqrt{\|p(T^*)\|_{0,0}\|p(T)\|_{0,0}} \leqslant (1+|\alpha| )\max\{1,1/(1-q)\}\|T\|.
$$
\end{proof}

\begin{remark}In order to keep the essentially self-adjoint operators, we should make an additional assumption on the family of operators $\set{T_j}_{j=1}^n$ see \cite[Subsection 2.5]{Ans01}, but we do not, because the further computation can be done purely algebraically. \textbf{In the remainder of the article we assume that the operator $T$ is bounded and self-adjoint.} We notice that in fact we have proved a more general Proposition \ref{Prop:samosprzezone}, because in a forthcoming paper we are going to investigate L\'evy processes of type B.
\end{remark}

\subsection{Partitions of type B}Let $[n]$ be the set $\{1,\dots,n\}$. For an ordered set $S$, denote by $\Part(S)$ the lattice of set partitions of that set. We write $B \in \pi$ if $B$ is an element of $\pi$ and we say that $B$ is a \emph{block of $\pi$}. A block of $\pi$ is called a \emph{singleton} if it consists of one element, and let $Sing(\pi)$ be the set of singletons of a set partition $\pi$. Given a partition $\pi$ of the set $[n]$, we write $\Semi(\pi)$ for the set of pairs of integers $(i, j)$ which occur in the same block of $\pi$ such that $j$ is the smallest element of the block greater than $i$. The same notation $\Semi(B)$ is applied to a block $B\in \pi$. Thus, when we draw the points of a block then we think that consecutive elements in every block (bigger than one) are connected by arcs above the $x$ axis -- see Figure \ref{Polkola}.
   
\begin{figure}[ht]
\begin{center}

\includegraphics[width=0.7\textwidth]{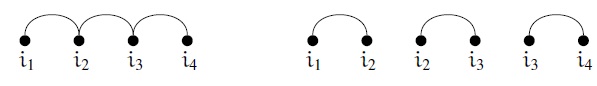}
\caption{ The example of a block and corresponding arcs.
}
\label{Polkola}
\end{center}
\end{figure}

\begin{definition}We call $\pi_f$ a \emph{type-B set partition of $[n]$} if $\pi$ is a set partition of $[n]$ and $f:\Semi(\pi)\cup \Sing(\pi)\to \{\pm 1\}$ is a coloring of the singletons or arcs. We denote by $\PB(n)$ the set of all type-B partitions of $[n]$, such that 
\begin{enumerate}
\item each singleton is necessarily colored by $1$;
\item each arc is colored by $\pm 1$ (see Figure \ref{RestrictedCrossingsColorxxx}).
\end{enumerate}
\end{definition}
\begin{figure}[ht]
\begin{center}
\includegraphics[width=0.8\textwidth]{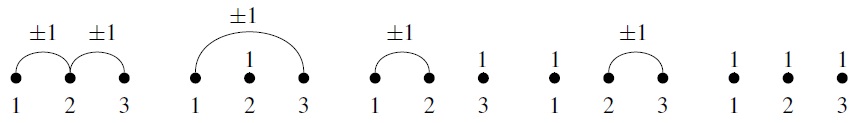}
\caption{$\PB(3)$. } \label{RestrictedCrossingsColorxxx}
\end{center}
\end{figure} 

When $n$ is even, we call $\pi_f \in \PB(n)$ a \emph{type-B pair partitions of $[n]$} if $\pi$ is a pair partition, i.e. each block consists of one arc. The set of type-B pair partitions of $[n]$ is denoted by $\PB_2(n)$ and the set of type-B partitions without singletons of $[n]$ is denoted by $\PB_{\geq 2}(n)$. 
\begin{remark}
\begin{enumerate}[\rm(1)]
\item Our notation $\pi_f$ should be understood in the sense  that $\pi_f=(\pi,f)$, where $\pi$ is a set partition of $[n]$ and $f$ is a coloring.
\item Our definition of set partitions of type B is different from \cite{BGN03,ChV06,R97,Sim00,RS10}. 
\item The set of all partitions $\P(n)$ is a subset of $\PB(n)$. The relationship between them can be written as $\pi\in \P(n)\cap \PB(n) \iff $ all arcs have color $1$.
\item In some sense, the above definition of type-B partitions is compatible with the groups $\Sigma(n)$. The Coxeter group of type B can be written as $\Sigma(n)=\Z_2^n \rtimes S(n) $ and hence it can be defined as all signed permutations of the numbers $\pm1,\dots,\pm n.$ Thus $\Sigma(n)$ consists of all signed permutations with signed entries in their window notation, i.e. we can assign arbitrary signs $\pm1$ to the points $1,2,\dots,n$. Hence, the above construction of the colorings of type-B partitions is similar to that of signed permutations, because we assign color $\pm 1$ to every arc. 
\end{enumerate}
\end{remark}
\noindent
\textbf{Blocks of type B.} Let $\pi_f\in\PB(n)$, then we denote by $ B_c\in \pi_f$ each colored block of type B, where $B=\{i_1,\dots,i_m\}$ and $c=(c_{1},\dots,c_{{m-1}})$, $c_{j}\in \{\pm 1\}$ is the color of arc $\{i_j,i_{j+1}\}$ and 
\begin{align*}&\Semi(B_c)=\big\{\{i_j,i_{j+1}\}_{c_{j}}\mid j\in[m-1]\big\}, \quad (m\geq 2),
\intertext{and} 
&\Semi(\pi_f)=\cup_{B_c\in\pi_f}\Semi(B_c).
\end{align*}
If $B$ is a singleton, then sometimes we do not write its color, i.e. $B:=B_{(1)}$, because each singleton is necessarily colored by $1$ and of course the same remark applies to $\Sing(\pi):=\Sing(\pi_f)$. To be clear, we also denote by $\max(B_c):=\max(B)$ its last element and by $\min(B_c):=\min(B)$ its first element. We order the blocks of $\pi_f=\{B^1_{c_1},\dots,B^l_{c_l}\}$ according to the order of their last elements, i.e. $\max(B^1_{c_1}) < \max(B^2_{c_2}) < \ldots < \max(B^l_{c_l})$.

\subsection{Statistics}Now, we introduce some partition statistics. Let $ B_{c}, \widetilde{B}_{\widetilde{c}}\in \pi_f $ with $\# B_{c},\# \widetilde{B}_{\widetilde{c}} \geq 2$, where $\pi_f\in \PB(n)$. 

\noindent \textbf{\emph{Restricted crossings}}. Note that in the definition below the colorings of arcs are not important. Thus we use the same definition of \emph{restricted crossings} as given in Biane \cite{Bia97}. We say that the arc $\{i,j\}_c$ \emph{is crossing} the arc $\{i',j'\}_{c'}$ if  $i < i' < j < j'  \textrm{ or } i' < i < j' < j $. Now we can define
\begin{align*}
\rc(B_{c},\widetilde{B}_{\widetilde{c}})
= &\# \big\{(V,W)\in \Semi(B_{c})\times \Semi(\widetilde{B}_{\widetilde{c}}) \mid
 \textrm{ such that } V \textrm{ is crossing } W \big\}.
\end{align*}
The number of restricted crossings of $\pi$ is $$\rc{(\pi)}=\rc{(\pi_f)} = \sum_{i<j} \rc(B^i_{c_i}, B^j_{c_j}),$$ where $\pi_f\setminus \Sing(\pi)=\{B^1_{c_1},\dots,B^l_{c_l}\}$.

\noindent \textbf{\emph{Restricted negative nestings}}. 
 Now we define the number of \emph{restricted negative nestings} of the partition $\pi_f\in\PB(n)$. We say that an arc $\{i,j\}_c$ \emph{nests} $\{i',j'\}_{c'}$ if $i <k <j$ for any $k\in \{i',j'\}_{c'}$. The set of nestings of $B_{c},\widetilde{B}_{\widetilde{c}}$ is
\begin{align*} \nest(B_{c},\widetilde{B}_{\widetilde{c}})=&\big\{(V,W)\in \Semi(B_{c})\times \Semi(\widetilde{B}_{\widetilde{c}})\mid V \text{~nests~} W\textrm{ } or \textrm{ } W \text{~nests~} V\big\},
\end{align*}
and the set of nesting of $\pi$ is 
\begin{align*}\nest(\pi)=\nest(\pi_f) = \cup_{i<j} \nest(B^i_{c_i}, B^j_{c_j}),\end{align*} 
where $\pi_f\setminus \Sing(\pi)=\{B^1_{c_1},\dots,B^l_{c_l}\}$. The number of \emph{restricted negative nestings} is defined by 
\begin{align*}
\InNB(\pi_f)=&\#\{ (V,W)\in \nest(\pi) \mid  V \text{~nests~} W \textrm{ and } f(W)=-1
\textrm{ or }\textrm{ } W \text{~nests~} V \textrm{ and } f(V)=-1\}.
\end{align*}

\noindent \textbf{\emph{Negative arcs}}. Let $\Semi(\pi_f,-1) :=\{W\in\Semi(\pi_f)\mid f(W)=-1\}$, then 
$$\NB(\pi_f)=\#\Semi(\pi_f,-1),$$
is the number of negative arcs. 
\begin{remark} In the above equation, we skip the index $f$ in $\pi$ if our statistic does not depend on coloring. 
\end{remark}

\section{A deformed probability of type B}
\subsection{Operators and cumulants of type B}In this non-commutative setting, random variables are understood to be the elements of the $*$-algebra generated by $\{\B(x),\B^{*}(x),p(T_x)\mid x\in H_\R\}$. Particularly interesting are their joint mixed moments, i.e. expressions
$$\state\big( [\B(x_n) +\B^\ast(x_n)+p(T_{x_n})]\dots [\B(x_1) +\B^\ast(x_1)+p(T_{x_1})]\big).$$
In order to work effectively on this object we need to define corresponding cumulants. This topic in the case of $q$-deformed Fock space was deeply analyzed in the literature; see \cite{Nic95,Nic96,Ans01,L05}. One of the first definitions of cumulants appropriate for $q$-deformed probability theory has already been given in \cite{Nic95}, based on an analog of the canonical form introduced by Voiculescu in the context of free probability. The advantage of the approach of that paper is that Nica's cumulants are defined for any probability distribution  all of whose moments are finite, however, the canonical form in that paper is not self-adjoint. 
Also Lehner \cite{L05} developed some general formulas, computed cumulants and partitioned cumulants for generalized Toeplitz operators. Our approach is close to \cite{HudPar,SchurCondPos,Ans01}. We provide an explicit formula for the combinatorial cumulants, involving the number of restricted crossings and negative nestings of a partition. First, we need to define the operators and cumulants, where we asume that $T_x$ are fixed bounded self-adjoint operators on $H$ indexed by $x\in H_\R$.
\begin{Defn}The operator 
\begin{equation}
\X(x):= \B(x) +\B^\ast(x)+p(T_x)+\lambda I,\qquad x \in H_\R, \quad\lambda\in\R,
\end{equation}
on $\F$ is called  \emph{ operator of type B} and $\mathcal{B}_{\alpha,q}:=\mathcal{B}_{\alpha,q}^0$. 
\end{Defn}
\begin{Defn}\label{Defi:Cumulant} Let $\pi_f\in \PB(n)$, $B_c=\{i_1,\dots,i_m\}_{(c_{1},\dots,c_{{m-1}})} \in \pi_f$, 
 $\mb{x}=(x_1,\dots,x_n) \in H_\R^n$
 and $\lambda_i\in \R$. \emph{The deformed cumulant of type B} is defined by
\begin{align*}
&\Cum({B_c}) := 
\begin{cases}
 \lambda_{\min(B_c)} & \text{ if } \#B_c=1, \\
\langle x_{\max(B_c)}, \c_{{m-1}}T_{x_{i_{m-1}}} \dots \c_{3}T_{x_{i_3}}\c_{2}T_{x_{i_2}} \c_{1}x_{\min(B_c)} \rangle & \text{ if } \#B_c\geq 2,
\end{cases}
\\
&\Cum_{\pi_f}:=\prod_{B_c \in \pi_f}\Cum({B_c}),
\end{align*}
where $\c_i$ is a color operator $\c_i:H \mapsto H$ defined as follows: 
\[
\c_i(x) = 
\begin{cases}
x & \text{ if } c_i=1, \\
\overline{x} & \text{ if } c_i=-1.
\end{cases}
\]
\end{Defn}

\noindent
The following theorem is the main result of the paper. Its proof is given in Section 3.3. 
\begin{theorem}\label{thm2} Suppose that $\mb{x}=(x_1,\dots,x_n) \in H_\R^n$, then 
\begin{align} 
\state\big( \mathcal{B}^{\lambda_n}_{\alpha,q}(x_n)\cdots \mathcal{B}^{\lambda_1}_{\alpha,q}(x_{ 1})\big)=
\displaystyle\sum_{\pi_f\in \PB(n)} \alpha^{\NB(\pi_f)}q^{\rc(\pi)+2 \InNB(\pi_f)}\Cum_{\pi_f}. \label{wick:glowna}
\end{align}
\end{theorem}

\begin{Cor}\label{cor13} Assume that $x_i \in H_\R$ for $i\in[n]$.
\begin{enumerate}[\rm(1)]
\item For $\alpha=\lambda=0$ and $\overline{x_i}=x_i$, we obtain the $q$-deformed formula for moments of random variable on $q$-Fock space (see \cite{Ans01} or \cite[Proposition 6]{Ans04})
$$
\state_{0,q}\big( \mathcal{B}_{0,q}(x_n)\cdots \mathcal{B}_{0,q}(x_{ 1})\big)
=\sum_{\pi \in \P_{\geq 2}(n)} q^{\rc(\pi)}\prod_{B \in\pi }\Big\langle x_{\max(B)}, \prod_{i\in B, i\neq \min(B),\max(B)} T_{x_{i}}x_{\min(B)}\Big\rangle. 
$$
\item For $T=\mathbf{0}$ and $\lambda=0$ we get the formula for Gaussian operator of type B \cite[Corollary 3.9]{BEH15} 
\begin{align*}
\state\big( \mathcal{B}_{\alpha,q}(x_n)\cdots \mathcal{B}_{\alpha,q}(x_{ 1})\big)=\sum_{\pi_f \in \PB_{ 2}(n)} \alpha^{\NB(\pi_f)}q^{\rc(\pi)+2 \InNB(\pi_f)} \prod_{\substack{\{i,j\} \in \pi_f \\ f(\{i,j\})=1} }\langle x_i, x_j\rangle \prod_{\substack{\{i,j\} \in \pi_f\\ f(\{i,j\})=-1}} \langle x_i, \overline{x_j}\rangle. 
\end{align*}
\item For $q=\lambda=0$ and $\overline{x_i}=x_i$, we get 
\begin{align*}
\state_{\alpha,0}\big( {\mathcal{B}_{\alpha,0}}(x_n)\cdots {\mathcal{B}_{\alpha,0}}(x_{1})\big) 
&=\sum_{\pi \in \NC_{\geq 2}(n)} (1+\alpha)^{\#OutArc(\pi)}\Big\langle x_{\max(B)}, \prod_{i\in B, i\neq \min(B),\max(B)} T_{x_{i}}x_{\min(B)}\Big\rangle.
\end{align*}
where \begin{itemize}
\item $\#OutArc(\pi)$ is the number of outer arcs in $\pi$ in a sense that these arcs are not nested by others;
\item $\NC_{\geq 2}(n)$ is the set of noncrossing partitions of $[n]$ without singletons. 
\end{itemize}
\end{enumerate}
\end{Cor}
\begin{proof}(1) If we put $\alpha=\lambda=0$ in \eqref{wick:glowna}, then all arcs have color $1$ and 
$$
\Cum(B_c) = 
\Big\langle x_{\max(B)}, \prod_{i\in B, i\neq \min(B),\max(B)} T_{x_{i}}x_{\min(B)}\Big\rangle  \text{ if } c=(1,\dots,1).$$
\\(2) Under the assumption that $T=\mathbf{0}$ and $\lambda=0$, we see
$$\Cum({B_c})=\langle x_{\max(B_c)}, \c_{{m-1}}T_{x_{i_{m-1}}} \dots \c_{3}T_{x_{i_3}}\c_{2}T_{x_{i_2}} \c_{1}x_{\min(B_c)} \rangle \neq 0$$ $\iff$ $\#B_c=2$, which means that $\pi_f\in \PB_{ 2}(n).$
\\(3) When $q=\lambda=0$, then the nonzero terms in part (2) of Theorem \ref{thm2} occur only when $\pi$ is noncrossing and inner arcs are colored by 1. 
Thus, for $\pi \in \NC_{\geq 2}(n)$ we count how many times we can assign color $-1$ to outer arcs and so we get 
$$\sum_{i=0}^{\#OutArc(\pi)}\alpha^i{\#OutArc(\pi) \choose i}=\left(1+\alpha\right)^{\#OutArc(\pi)}.$$ 
\end{proof}
\begin{example}The deformed cumulants and partitions of type B for $\state\big(\mathcal{B}_{\alpha,q}(x_4)\mathcal{B}_{\alpha,q}(x_3)\mathcal{B}_{\alpha,q}(x_2)\mathcal{B}_{\alpha,q}(x_1)\big)$ can be graphically represented in Figure \ref{LiczenieMomentow}.
\begin{figure}[h]
\begin{center}
\includegraphics[width=0.7\textwidth]{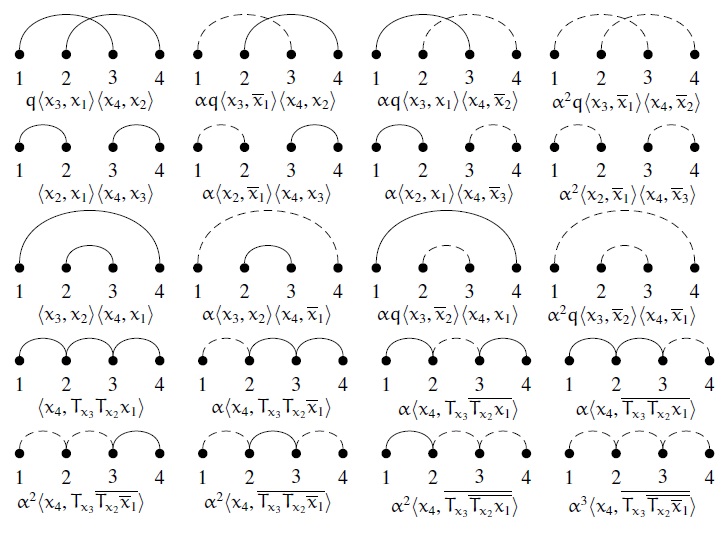}

\caption{Deformed cumulants, statistics and partitions of type B for  $\state\big(\mathcal{B}_{\alpha,q}(x_4)\mathcal{B}_{\alpha,q}(x_3)\mathcal{B}_{\alpha,q}(x_2)\mathcal{B}_{\alpha,q}(x_1)\big)$.}
 \label{LiczenieMomentow}
\end{center}
\end{figure} 
\end{example}

\subsection{The Poisson distribution of type B}
\label{Subsec:Poisson}
For a probability measure $\mu$ with finite moments of all orders, let us orthogonalize the sequence $(1,y,y^2,y^3,\dots)$ in the Hilbert space $L^2(\R,\mu)$, following the Gram-Schmidt method. This procedure yields orthogonal polynomials $(P_0(y), P_1(y), P_2(y), \dots)$ with $\text{deg}\, P_n(y) =n$. Multiplying by constants, we take $P_n(y)$ to be monic, i.e. the coefficient of $y^n$ is 1. It is known that they satisfy a recurrence relation
\[
y P_n(y) = P_{n+1}(y) +\beta_n P_n(y) + \gamma_{n-1} P_{n-1}(y),\qquad n =0,1,2,\dots
\]
with the convention that $P_{-1}(y)=0$. The coefficients $\beta_n$ and $\gamma_n$ are called \emph{Jacobi parameters} and they satisfy $\beta_n \in \R$ and $\gamma_n \geq 0$. 
The following representation of type-B Poisson distribution is in the spirit of \cite{Ans01} and rather different from that of \cite{SaiKraw,SaiPoisson}.
\begin{Defn}
$(\alpha,q)$-Poisson of type B polynomials are defined by the recursion relations
\begin{align}
\label{Charlier}
y P_{ n}^{(\alpha,q)}( y) &= P_{ n+1}^{(\alpha,q)}(y) + [n]_q (1+ \alpha q^{n-1}) P_{ n}^{(\alpha,q)}( y) + [n]_q (1 + \alpha q^{n-1}) P_{n-1}^{(\alpha,q)}( y), \qquad n\geq 1
\end{align}
with initial conditions $P_{-1}^{(\alpha,q)}(y) = 0$, $P_{0}^{(\alpha,q)}( y) = 1$ and $P_{1}^{(\alpha,q)}( y) = y$. There exists a probability measure $\mu_{\alpha,q}$ which is associated to the orthogonal polynomials $P_{ n}^{(\alpha,q)}$. 
\end{Defn}
\begin{remark} Professor M. Ismail informed us that the measure of orthogonality of the above polynomial sequence is not known. The $(\alpha,q)$-Poisson of type B polynomials are $q$-analogues of the polynomials studied in \cite{IsmaKo12} (equation (5.11) with $\nu =1$). In special cases, we can identify this measure:
\begin{enumerate}[\rm(1)]
\item the measure $\mu_{\alpha,1}$ is the classical Poisson law; 
\item the measure $\mu_{0,0}$ is the Marchenko-Pastur distribution; 
\item the measure $\mu_{0,q}$ is the $q$-Poisson law and the orthogonal polynomials $P_n^{(0,q)}(y)$ are called \emph{$q$-Poisson-Charlier polynomials} (see \cite{Ans01}); 
\item the measure $\mu_{\alpha,-1}$ is a non-symmetric Bernoulli distribution; 
\item the measure $\mu_{\alpha,0}$, $\alpha\neq 0$ is a two-state free Meixner distribution because its Jacobi parameters are independent of $n$ for $n\geq 2$ (see \cite{AnsMlot12,DGIX09}). The corresponding measure belongs to the Bernstein-Szeg$\ddot{o}$ class i.e. has at most 3 atoms and absolutely continuous part of $\mu_{\alpha,0}$ is  
$$\frac{\sqrt{4-(x-1)^2}}{p(x)}dx
,$$
where $p(x)$ is a cubic polynomial.
\end{enumerate}
\end{remark}
\begin{Prop} Suppose that $\alpha,q\in(-1,1)$ and $x \in H, \|x\|=1, \overline{x}=\pm x$ and $T=Id$. Then the probability distribution of $\mathcal{B}_{\alpha,q} $ with respect to the vacuum state is given by $\mu_{\alpha,q}$ if $\overline{x}= x$ and $\mu_{-\alpha,q}$ if $\overline{x}= -x$. 
\label{WielomianyTypeB}
\end{Prop}
\begin{proof}First assume that $\overline{x}= x$. Note that for $n=1$ $P_{1}^{(\alpha,q)}(\mathcal{B}_{\alpha,q}(x))\Omega =\mathcal{B}_{\alpha,q}(x)\Omega=x$ and by induction 
\begin{align*}
&P_{ n+1}^{(\alpha,q)}( \mathcal{B}_{\alpha,q}(x))\Omega =\mathcal{B}_{\alpha,q}(x) P_{ n}^{(\alpha,q)}( \mathcal{B}_{\alpha,q}(x))\Omega\\&-[n]_q (1+ \alpha q^{n-1}) P_{ n}^{(\alpha,q)}( \mathcal{B}_{\alpha,q}(x))\Omega - [n]_q (1 + \alpha q^{n-1}) P_{n-1}^{(\alpha,q)}( \mathcal{B}_{\alpha,q}(x))\Omega\\&=\mathcal{B}_{\alpha,q}(x) x^{\otimes n}-[n]_q (1 + \alpha q^{n-1})x^{\otimes n}- [n]_q(1+ \alpha q^{n-1})x^{\otimes (n-1)}=x^{\otimes n+1}.
\end{align*}
Hence, from above it follows that
\begin{equation}
\|x^{\otimes n}\|_{\alpha,q}= \|P_n^{(\alpha,q)}\|_{L^2},\qquad n\in\N\cup\{0\}. 
\end{equation}
Therefore the map $\Phi\colon (\text{span}\{x^{\otimes n}\mid n \geq 0\}, \|\cdot\|_{\alpha,q}) \to L^2(\R,\mu_{\alpha,q})$ defined by $\Phi(x^{\otimes n})= P_n^{(\alpha,q)}(y)$ is an isometry.  Since $\Phi$ is an isometry, we get $\langle \Omega,\mathcal{B}_{\alpha,q}^n(x)\Omega\rangle_{\alpha,q} = m_n(\mu_{\alpha,q})$. 
By Proposition \ref{Ansh+} we conclude that $\mathcal{B}_{\alpha,q}(x)$ is bounded, so its vacuum distribution is compactly supported. Hence, the moments uniquely determine the measure and we conclude that $\mathcal{B}_{\alpha,q}(x)$ has the distribution $\mu_{\alpha,q}$. We proceed analogously if $\overline{x}=-x$. 
\end{proof}

\subsection{Proof of the main theorem}We begin with some special notations.

\noindent
\textbf{Extended partition.} In order to prove the main theorem we need the set $\PB_E(n)$ of so-called extended partitions, which contains the set $\PB(n)$. We use these partitions in the proof of Theorem \ref{twr:momentogolne}, only. Here each block of size at least two can be additionally marked by $\prime$. More precisely, for the fixed $\pi_f \in \PB(n)$, and $B_c\in \pi_f$, where $\#B_c\geq 2$, we consider additional numbers $\bar{1},\bar{2},\dots,\bar{n}$ and define  $B_c'=\{i_1,\dots,i_m\}'_{c}:=\{\bar{i}_1,\dots,\bar{i}_m\}_{c}.$
\begin{Defn}The set of partitions $\PB_E(n)$ is defined from the set $\PB(n)$, as follows:   
$$ \PB_E(n) = \big\{S'\cup (\pi_f \setminus S)\mid S\subset \pi_f \setminus \Sing(\pi), \textrm{ where }\pi_f\in \PB(n)\big\},$$ 
where $S':=\{A'\mid A\in S\}$ and $\emptyset'=\emptyset.$ 
\end{Defn}
 \begin{remark} We can think about operation $\prime$ in this way, that we pick some  blocks of size at least two and marked them by $\prime.$ 
\end{remark}
\begin{example}For example,
\begin{align*} \PB_E(3)=&\Big\{\big\{\{1\},\{2\},\{3\}\big\},\big\{\{1,2\}_{(\pm1)},\{3\}\big\},\big\{\{1,3\}_{(\pm1)},\{2\}\big\},\big\{\{2,3\}_{(\pm1)},\{1\}\big\}, \big\{\{1,2,3\}_{(\pm1,\pm1)}\big\},\\&\big\{\{1,2\}\primes_{(\pm1)},\{3\}\big\},\big\{\{1,3\}\primes_{(\pm1)},\{2\}\big\},\big\{\{2,3\}\primes_{(\pm1)},\{1\}\big\},\big\{\{1,2,3\}\primes_{(\pm1,\pm1)}\big\}\Big\}.\end{align*}
\end{example}
\noindent 
For $\pi_f \in \PB_E(n) $ we denote by $Block\primes(\pi_f)$ the blocks of $\pi_f$ which are marked by $\prime$ and $Block(\pi_f)=\pi_f\setminus \big(Block\primes(\pi_f)\cup \Sing(\pi)\big).$ Thus we can decompose an extended partition of type B as a disjoint subset $$\pi_f=Block\primes(\pi_f)\cup Block(\pi_f)\cup \Sing(\pi).$$
\noindent \textbf{\emph{Cover and left of max}}. We also need some other statistics: $\InS(\pi)$ and $\SLNB(\pi_f)$ for $\pi_f \in \PB_E(n)$ which we will use in Theorem \ref{twr:momentogolne}, only. Namely, we define
\begin{align*}
&\InS(\pi)=\#\big\{ (V,W)\in \big[Block\primes(\pi_f)\cup \Sing(\pi)\big]\times \Semi(\pi_f) \mid 
 i<\max(V)<j \textrm{ for } i,j\in W\big\},
\\&
\SLNB(\pi_f)=\#\{ (V,W)\in \big[Block\primes(\pi_f)\cup \Sing(\pi)\big]\times \Semi(\pi_f,-1) \mid \max(V) < j \text{~for all~} j\in W \}.
\end{align*}
\begin{remark} Note that  $\InS(\pi)$ represents the number of covered singletons and $\SLNB(\pi_f)$  the number of singletons to the left of negative arcs, whenever we have $Block\primes(\pi_f)=\emptyset.$
\end{remark}
\begin{example} For the partition in Figure \ref{RestrictedCrossings1} we see that 
\begin{itemize}
\item if $\pi = \{\{1,4,6,7\}_{(-1,1,-1)},\{2\},\{3,5,10\}\primes_{(1,-1)},\{8,12\}_{(-1)},\{9,11\}_{(1)}\}$, then $\rc{(\pi)}=5$, $\InNB(\pi_f)=1$, $\NB(\pi_f)=4$, $\InS(\pi)=3$ and $\SLNB(\pi_f)=3$;
\item if $\pi = \{\{1,4,6,7\}_{(-1,1,-1)},\{2\},\{3,5,10\}_{(1,-1)},\{8,12\}_{(-1)},\{9,11\}_{(1)}\}$, then $\rc{(\pi)}=5$, $\InNB(\pi_f)=1$, $\NB(\pi_f)=4$,  $\InS(\pi)=1$ and $\SLNB(\pi_f)=3$; 
\item if $\pi = \{\{1,4,6,7\}\primes_{(-1,1,-1)},\{2\},\{3,5,10\}_{(1,-1)},\{8,12\}_{(-1)},\{9,11\}_{(1)}\}$, then $\rc{(\pi)}=5$, $\InNB(\pi_f)=1$, $\NB(\pi_f)=4$, $\InS(\pi)=2$ and $\SLNB(\pi_f)=4$. 
\end{itemize}

\begin{figure}[ht]

\begin{center}

\includegraphics[width=0.7\textwidth]{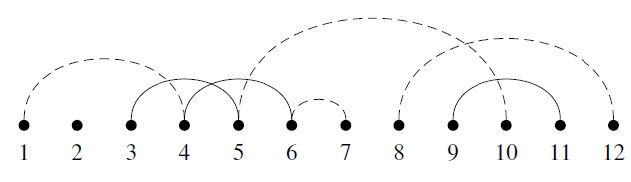}

\caption{A partition of $12$ elements with  five blocks 
}
 \label{RestrictedCrossings1}
\end{center}
\end{figure} 
\end{example}

\noindent
In order to simplify  notation, we define the following operators, which map $H$ into $H$ and which are indexed by the block $B_c=\{i_1,\dots,i_m\}_{(c_{1},\dots,c_{{m-1}})}\in \pi$, i.e.
\begin{align*}&\OptTTyl_{B_c}=
\begin{cases}
I & \text{ if } B_c\in \Sing(\pi)\text{ } (c=1)
\\
T_{x_{i_m}}\c_{{m-1}} \dots \c_{3}T_{x_{i_3}}\c_{2}T_{x_{i_2}} \c_{1}& \text{ if } B_c\in Block\primes(\pi_f), 
\end{cases}
\\
&\OptT_{B_c}=
\begin{cases}\c_{1} & \text{ if } \#B_c=2 \text{ and } B_c\in Block(\pi_f),
\\
\c_{{m-1}}T_{x_{i_{m-1}}} \dots \c_{3}T_{x_{i_3}}\c_{2}T_{x_{i_2}} \c_{1}& \text{ if } \#B_c>2 \text{ and }B_c\in Block(\pi_f), 
\end{cases}
\end{align*}
where $\mb{x}=(x_1,\dots,x_n) \in H_\R^n$, $\pi_f\in\PB_{E}(n)$ and $f_i$ is a color operator introduced in the Definition \ref{Defi:Cumulant} of cumulants. With the above notations we also introduce 
\begin{align*}
\begin{split}
&\Cumm=\prod_{\substack{B_c \in \pi_f\\ B_c\in Block(\pi_f)}} \langle x_{\max(B_c)}, \OptT_{B_c}x_{\min(B_c)}\rangle, \\ 
&\CummTensor=\bigotimes_{\substack{ B_c \in \pi_f\\ B_c\in Block\primes(\pi_f)\cup \Sing(\pi) }}\Bigg\{ \OptTTyl_{B_c} x_{\min(B_c)}\Bigg\}_{\max(B_c)}.
\end{split}
\end{align*}
Notice that in the above formula we use the following bracket notation $\{\star\}_{\max(B_c)}$, which should be understood that the position of $\star$ (in the tensor product) is ordered with respect to the $\max(B_c)$. For example, if $Block\primes(\pi_f)\cup \Sing(\pi)=\{\{4\},\{6\},\{2,5,7\}\primes_{(1, -1)},\{1,3\}\primes_{(-1)}\} $, then 
$$\CummTensor=T_{x_3}\overline{x}_1\otimes x_4\otimes x_6 \otimes T_{x_7}\overline{T_{x_5}x_2}.$$
We also use the following conventions for $\epsilon\in\{1,\ast,\prime\}$
\[
\B^{\epsilon}(x) = 
\begin{cases}
\B^{*}(x) & \text{ if } \epsilon=\ast, \\
\r_q(x)+ \alpha \ell^N_{q}({x}) & \text{ if } \epsilon=1,
\\
\r_q^{T_x}(x)+ \alpha \ell_{q}^{N,T_x}({x})& \text{ if } \epsilon=\prime.
\end{cases}
\]
\noindent
Now we prove the following theorem, which shows the relationship between partitions of type B (corresponding statistic) and a joint action of operators on a vacuum vector introduced in Section 2. This theorem is of
independent interest,  and will
also help in applications of Theorem \ref{thm2}.  
\begin{theorem}\label{twr:momentogolne}
For any $\mb{x}=(x_1,\dots,x_n) \in H_\R^n$ and any $\epsilon=(\epsilon(1), \dots, \epsilon(n))\in\{1,\ast,\prime\}^n$, we have
\begin{equation}\label{formula101}
\B^{\epsilon(n)}(x_{n})\cdots \B^{\epsilon(1)}(x_1)\Omega = \sum_{\pi_f\in\PB_{E;\epsilon}(n)} \alpha^{\NB(\pi_f)}q^{\rc(\pi) +\InS(\pi)+ 2 \InNB(\pi_f)+2\SLNB(\pi_f) } \Cumm\CummTensor.
\end{equation}
\end{theorem}
\begin{remark} (1). If $Block\primes(\pi_f)\cup \Sing(\pi)=\emptyset$, then $\CummTensor=\Omega.$ 

\noindent (2). If $\#\{i\in[j] \mid \epsilon(i)=1\}>\#\{i\in[j] \mid \epsilon(i)=\ast\}$ 
 for some $j\in[n]$, then we have
$$\B^{\epsilon(n)}(x_{n})\cdots \B^{\epsilon(1)}(x_1)\Omega=0.$$
This case is also covered by \eqref{formula101} if we understand that the sum over the empty set is 0 since $\PB_{E;\epsilon}(n)=\emptyset$ in this case. 

\noindent (3).  The main new ingredient in the proof of Theorem~3 appears in step~3, 
besides some new concepts point~2. 
In step 2a) and 2b) a new block (cumulant) is created by adding arcs, 
while in the proof of \cite[Theorem 3.7]{BEH15}  just a pair is created.
Steps 3a) and 3b) treat the gauge operator and during this
procedure we generate extended blocks, which are necessary to produce
higher order cumulants.
During these steps calculating the change in the statistic a bit more
complicated than in the proof of \cite[Theorem 3.7]{BEH15} and in fact
somewhat unnatural, because of some fairly intricate dependencies
between the various type of blocks.
\end{remark}
\begin{proof} The proof is given by induction. When $n=1$, $\B(x_1)\Omega=p(T_{x_1})\Omega=0$ and $\B^\ast(x_1)\Omega=x_1$ and hence the formula is true. Suppose that the formula is true for $n=k$. Then for any $\epsilon\in\{1,\ast,\prime\}^{k}$, we get 
\begin{align*}
\B^{\epsilon(k)}(x_{k})\cdots \B^{\epsilon(1)}(x_1)\Omega = \sum_{\pi_f\in\PB_{E;\epsilon}(k)} \alpha^{\NB(\pi_f)}q^{\rc(\pi) +\InS(\pi)+ 2 \InNB(\pi_f)+2\SLNB(\pi_f) } \Cumm\CummTensor.
\end{align*}
We will show that the action of $\B^{\epsilon(k+1)}(x_{k+1})$ corresponds to the inductive graphic description of set partitions of type B.
\noindent
From now on, we fix a partition $\pi_f\in \PB_{E;\epsilon}(k)$ and one block $ \BC\in Block\primes(\pi_f)\cup \Sing(\pi)$ in the $\ith^{\rm th}$ position i.e. $\ith=\max( \BC)$. In this situation, the block $\BC$ contributes to the element $\Tens_{\BC} x_{\min(\BC)}$ (in a sense that $\CummTensor=\dots\otimes \{\Tens_{\BC} x_{\min(\BC)}\}_{\ith}\otimes \dots$). Suppose that $\pi_f$ has
\begin{itemize}
\item $p$ blocks $ B_c\in Block\primes(\pi_f)\cup \Sing(\pi)$, such that $ \max(B_c)<\ith$ -- we call them $Left(\BC)$,
\item $r$ blocks $B_c\in Block\primes(\pi_f)\cup \Sing(\pi)$, such that $ \max(B_c)> \ith $ -- we call them $Right(\BC)$. 
\end{itemize}
We understand that $p=0$ ($r=0$, respectively) when there are no blocks $B_c\in Block\primes(\pi_f)\cup \Sing(\pi)$ such that $\max(B_c)$ is to the left (right, respectively) of $\ith$. Here we assume more, that $\pi_f$ has 
\begin{itemize}
\item arcs $U_1,\dots, U_s$ with color $1$ to the right of $\ith^{\rm th}$ in the strict sense,
\item arcs $V_1,\dots, V_t$ with color $-1$ to the right of $\ith^{\rm th}$ in the strict sense,
\item arcs $W_1,\dots, W_u$ which cover $\ith^{\rm th}$ in a sense that $a_\star<\ith<b_\star$ for $a_\star,b_\star \in W_\star$.
\end{itemize}

\begin{figure}[ht]

\begin{center}

\includegraphics[width=0.7\textwidth]{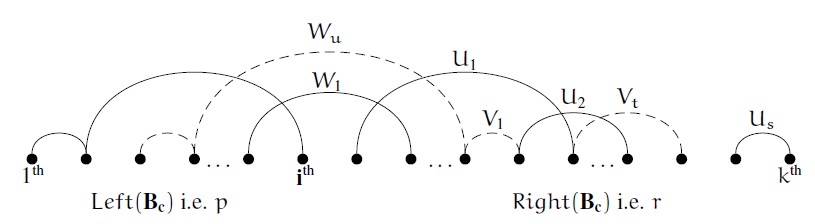}

\caption{ The main structure of partition $\pi_f\in \PB_{E;\epsilon}(k)$ in the induction step.  
 }
 \label{RestrictedCrossingsPoofUW}
\end{center}
\end{figure}
There may be arcs to the left of $\ith^{\rm th}$, but they do not matter -- see Figure \ref{RestrictedCrossingsPoofUW}. Note here that $p+r+1=\#\Sing(\pi)+\# Block'(\pi_f)$. In the proof we use the notation $\tilde{\textbf{B}}_{{(\textbf{c},\pm1 )}}$, which denotes the block created from $\BC$ by adding color $\pm 1$ to the last coordinate of $\textbf{c}$, i.e $(\textbf{c},\pm1 )$ and $\tilde{\textbf{B}}=\textbf{B} \cup \{k+1\}$. In the cases below, each subsequent operator $\B^{\epsilon}(x)$ contributes to 
\begin{itemize}
\item a new singleton in $\Sing(\tilde{\pi})$, $\text{ if } \epsilon=\ast$ -- Case 1;
\item a new block in $Block(\tilde{\pi}_{\tilde{f}})$, $\text{ if } \epsilon=1 $ -- Case 2;
\item a new block in $Block\primes(\tilde{\pi}_{\tilde{f}})$, $\text{ if } \epsilon=\prime$ -- Case 3.
\end{itemize} 
\noindent Case 1. If $\epsilon(k+1)=\ast$, then the operator $\B^\ast(x_{k+1})$ acts on the tensor product, putting $x_{k+1}$ on the right. This operation graphically corresponds to adding the singleton $\{k+1\}$ (with color 1) to $\pi_f\in\PB_{E;\epsilon}(k)$, to yield a new type-B partition $\tilde{\pi}_{\tilde{f}}\in\PB_{E;\epsilon}(k+1)$. This map $\pi_f\mapsto \tilde{\pi}_{\tilde{f}}$ does not change the numbers $\NB, \rc, \InNB, \SLNB$ or $\InS$, because a new singleton is the right most element of $\tilde{\pi}$. This is compatible with the fact that the action of $\B^\ast(x_{k+1})$ does not change the coefficient. Hence, the formula \eqref{formula101} holds when $n=k+1$ and $\epsilon(k+1)=\ast$. 
\\
\\
Now we move to Cases 2 and 3, where we assume $Block\primes(\pi_f)\cup \Sing(\pi)\neq \emptyset.$
\\
\\
Case 2. If $\epsilon(k+1)=1$, then we have two cases.
 
\noindent Case 2a. If $\r_q(x_{k+1})$ acts on the tensor product, then new $p+r+1$ terms appear by using \eqref{rq}. In the $\ith^{\rm th}$ term the inner product $\langle x_{k+1}, \Tens_{\BC} x_{\min(\BC)} \rangle$ appears with coefficient $q^r$. Graphically this corresponds to getting a set partition $\tilde{\pi}_{\tilde{f}} \in \PB_{E;\epsilon}(k+1)$ by adding $k+1$ to $\pi_f$ and creating the block $\tilde{\textbf{B}}_{\textbf{(C,1)}}\in Block(\tilde{\pi}_{\tilde{f}})$, i.e. now the last arc has color $1$ -- see Figure \ref{RestrictedCrossingsPoofPoint2}. This new arc $\{\ith,k+1\}$ 
\begin{itemize}
\item crosses the arcs $W_1, \dots, W_u$ and so increases the number of crossings by $u$;
\item decreases by $u$ the number of blocks $B_c\in Block\primes(\tilde{\pi}_{\tilde{f}})\cup \Sing(\tilde{\pi})$ such that $\max(B_c)$ is between the arcs $W_1,\dots, W_u$ (originally $\ith$ was between the arcs $W_1,\dots, W_u$);
\item increases by $r$ the number of blocks $B_c\in Block\primes(\tilde{\pi}_{\tilde{f}})\cup \Sing(\tilde{\pi})$ such that the $ \max(B_c)$ is covered by the new arc;
\item covers the negative arcs $V_1,\dots, V_t$;
\item causes that in the new situation $\BC$ is not a block in $\tilde{\pi}_{\tilde{f}}$, so the number $\SLNB(\pi_f)$ decreases by $t$. 
\end{itemize}

\begin{figure}[ht]

\begin{center}

\includegraphics[width=0.7\textwidth]{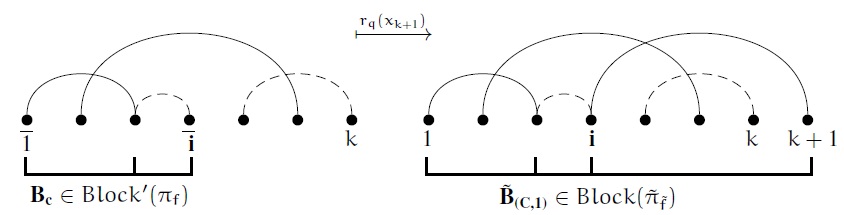}

\caption{  The visualization  of the action $\r_q(x_{k+1})$ on $\Tens_{\BC} x_{\min(\BC)}.$ 
}
 \label{RestrictedCrossingsPoofPoint2}
\end{center}
\end{figure}
Altogether we have: $\rc(\tilde{\pi})=\rc(\pi)+u$, $\InS(\tilde{\pi})=\InS(\pi)-u+r$, $\InNB(\tilde{\pi}_{\tilde{f}})=\InNB(\pi_f)+t$, $\SLNB(\tilde{\pi}_{\tilde{f}})= \SLNB(\pi_f)-t$ and $\NB(\tilde{\pi}_{\tilde{f}})=\NB(\pi_f)$. So the exponent of $q$ increases by $r$. Summarizing this, we get the factor $q^r$ and the inner product 
$$\left\langle x_{k+1},\Tens_{\BC} x_{\min(\BC)} \right\rangle =\left\langle x_{\max(\tilde{\textbf{B}}_{\textbf{(C,1)}})}, \OptT_{\tilde{\textbf{B}}_{\textbf{(C,1)}}} x_{\min(\tilde{\textbf{B}}_{\textbf{(C,1)}})} \right\rangle,$$
which is exactly the expression when $\r_q(x_{k+1})$ acts on $\Tens_{\BC} x_{\min(\BC)}$. 
\begin{remark} Note that in the above algorithm, we create a new pair in $Block(\tilde{\pi}_{\tilde{f}})$ with color $1$, whenever we have $\BC\in \Sing(\pi)$. An analogous remark applies to the remainder of the proof.
\end{remark}

\noindent Case 2b. If $\alpha \ell^N_{q}({x_{k+1}})$ acts on the tensor product, then new $p+r+1$ terms appear by using \eqref{lq}. In the $\ith^{\rm th}$ term the inner product $\left\langle x_{k+1}, \overline{\Tens_{\BC} x_{\min(\BC)}}\right\rangle$ appears with coefficient $\alpha q^{p+(p+r+1)-1}$. This means that we create a new block $\tilde{\textbf{B}}_{\textbf{(C,-1)}}\in Block(\tilde{\pi}_{\tilde{f}})$, i.e. the last arc $\{\ith,k+1\}$ has color $-1$ -- see Figure \ref{RestrictedCrossingsPoofPoint3}. Similarly to Case 2a, we calculate the changes of statistics and get: $\rc(\tilde{\pi})=\rc(\pi)+u$, $\InS(\tilde{\pi})=\InS(\pi)-u+r$, $\InNB(\tilde{\pi}_{\tilde{f}})=\InNB(\pi_f)+t$, $\SLNB(\tilde{\pi}_{\tilde{f}})= \SLNB(\pi_f)-t+p$, and $\NB(\tilde{\pi}_{\tilde{f}})=\NB(\pi_f)+1$. Altogether, when moving from $\pi_f$ to $\tilde{\pi}_{\tilde{f}}$, the exponent of $\alpha$ increases by $1$ and the exponent of $q$ increases by $2p+r$, which coincides with the coefficient appearing in the action of $\alpha \ell^N_{q}({x_{k+1}})$. In the end, we get the factor $q^{2p+r}$ and the inner product
$$\left\langle x_{k+1}, \overline{\Tens_{\BC} x_{\min(\BC)}} \right\rangle=\left\langle x_{\max(\tilde{\textbf{B}}_{\textbf{(C,-1)}})}, {\OptT_{\tilde{\textbf{B}}_{\textbf{(C,-1)}}} x_{\min(\tilde{\textbf{B}}_{\textbf{(C,-1)}})}}\right\rangle.$$
 
\begin{figure}[ht]

\begin{center}
\includegraphics[width=0.7\textwidth]{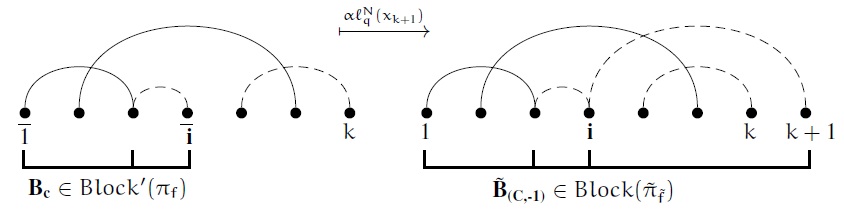}
\caption{ The visualization  of the action   $\alpha\ell^N_{q}({x_{k+1}})$ on $\Tens_{\BC} x_{\min(\BC)}$. 
 }
 \label{RestrictedCrossingsPoofPoint3}
\end{center}
\end{figure}

Case 3. If $\epsilon(k+1)=\prime$, then we have two cases. 
\noindent Case 3a. We use the equation \eqref{rqT}, delete the element $\Tens_{\BC} x_{\min(\BC)}$ from $\CummTensor$ 
 and then a new component in the tensor appears (in the last position $k+1$):
$$\bigotimes_{\substack{ B_c\in Block\primes(\pi_f)\cup \Sing(\pi)\\ B_c\neq \BC }} \Bigg\{\OptTTyl_{B_c}
x_{\min(B_c)}\Bigg\}_{\max(B_c)} \otimes \Bigg\{T_{x_{k+1}}\Big(\Tens_{\BC} x_{\min(\BC)})\Big)\Bigg\}_{k+1}$$
with coefficient $q^r$. Then we get a new partition $\tilde{\pi}_{\tilde{f}} \in \PB_{E;\epsilon}(k+1)$ by adding $k+1$ to $\pi_f$ and creating the block $\tilde{\textbf{B}}_{\textbf{(C,1)}}\in Block\primes(\tilde{\pi}_{\tilde{f}})$ (marked by $\prime$) with the last arc $\{\ith,k+1\}$ (with color $1$, see Figure \ref{RestrictedCrossingsPoofPoint2T}). Now $\max(\tilde{\textbf{B}}_{\textbf{(C,1)}})=k+1$, so we can calculate the change in the statistic generated by the new arc, analogously to Case 2a, because $\max(\tilde{\textbf{B}}_{\textbf{(C,1)}})$ cannot be covered or be to the left of some negative arc, i.e. the exponent of $q$ will increase by $r$. 
This situation is also compatible with changes inside the tensor product, i.e. $\Tens_{\tilde{\textbf{B}}_{\textbf{(C,1)}}} x_{\min(\tilde{\textbf{B}}_{\textbf{(C,1)}})}=T_{x_{k+1}}(\Tens_{\BC} x_{\min(\BC)})$.
 
\begin{figure}[ht]

\begin{center}

\includegraphics[width=0.7\textwidth]{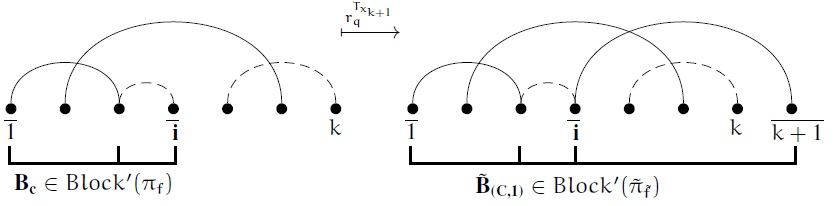}

\caption{  The visualization  of the action $\r_q^{T_{x_{k+1}}}$ on $\Tens_{\BC} x_{\min(\BC)}.$ 
}
 \label{RestrictedCrossingsPoofPoint2T}
\end{center}
\end{figure}

\noindent 
 Case 3b. We use equation \eqref{lqT}, then we get in the $\ith^{\rm th}$ term of the operator $\alpha \ell_{q}^{N,T_{{x_{k+1}}}}$
$$\bigotimes_{\substack{ B_c\in Block\primes(\pi_f)\cup \Sing(\pi)\\ B_c\neq \BC }} \Bigg\{\OptTTyl_{B_c}
x_{\min(B_c)}\Bigg\}_{\max(B_c)} \otimes \Bigg\{T_{x_{k+1}}\Big(\overline{ \Tens_{\BC} x_{\min(\BC)}}\Big)\Bigg\}_{k+1}$$ 
with coefficient $\alpha q^{p+(p+r+1)-1}$. Thus we obtain $\tilde{\pi}_{\tilde{f}} \in \PB_{E;\epsilon}(k+1)$ by adding $k+1$ to $\pi_f$ and create the block marked by $\prime$ with the last arc colored by $-1$, i.e. $\tilde{\textbf{B}}_{\textbf{(C,-1)}}\in Block\primes(\tilde{\pi}_{\tilde{f}})$ -- see Figure \ref{RestrictedCrossingsPoofPoint3T}. Similarly to the Case 3a $\max(\tilde{\textbf{B}}_{\textbf{(C,-1)}})=k+1$ holds, so we can use a change in statistic from Case 2b to get that the exponent of $\alpha$ increases by $1$, the exponent of $q$ increases by $2p+r$ and we have $\Tens_{\tilde{\textbf{B}}_{\textbf{(C,-1)}}} x_{\min(\tilde{\textbf{B}}_{\textbf{(C,-1)}})}=T_{x_{k+1}}\Big(\overline{ \Tens_{\BC} x_{\min(\BC)}}\Big)$.

\begin{figure}[ht]

\begin{center}
\includegraphics[width=0.7\textwidth]{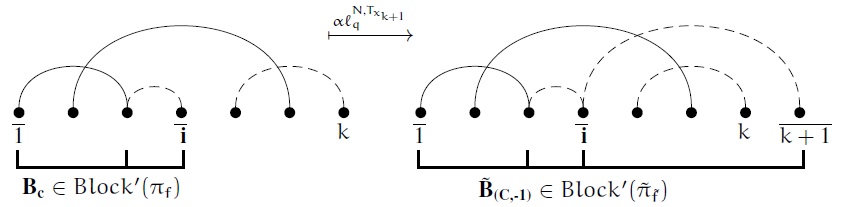}

\caption{ The visualization  of the action   $\alpha  \ell_{q}^{N,T_{{x_{k+1}}}}$ on $\Tens_{\BC} x_{\min(\BC)}$. 
 }
 \label{RestrictedCrossingsPoofPoint3T}
\end{center}
\end{figure}

\noindent Case 4. We have $Block\primes(\pi_f)\cup \Sing(\pi)=\emptyset$, then $\CummTensor=\Omega$ and the action of $\B(x)$ and $p(T_{x})$ on the vacuum vector gives zero, which is compatible with the fact that under this action we cannot create an arc from the $\Omega$. 
\\
\\ 
Note that as $\pi_f$ runs over $\PB_{E;(\epsilon(1),\dots,\epsilon(k))}(k)$, every set partition $$\tilde{\pi}_{\tilde{f}} \in \PB_{E;(\epsilon(1),\dots,\epsilon(k),\epsilon(k+1))}(k+1)$$ appears exactly once in one of Cases 1, 2, 3 or 4, which shows by induction that the formula \eqref{formula101} holds for all $n \in \N$. 
\end{proof}
\noindent
We now present the proof of Theorem \ref{thm2}.
\begin{proof} First, let us notice that for $\epsilon\in\{1,\ast,\prime\}^n$ we have
\begin{align} \label{wick:prawieglowna}
&\state\big( \B^{\epsilon(n)}(x_{n})\cdots \B^{\epsilon(1)}(x_1)\big)=\sum_{\pi_f \in \PB_{\geq 2,\epsilon}(n)} \alpha^{\NB(\pi_f)}q^{\rc(\pi)+2 \InNB(\pi_f)}
\Cumm\end{align}
Indeed, from equation \eqref{formula101} we see that the following condition must hold: $\CummTensor=\Omega$. This will happen if and only if $Block\primes(\pi_f)\cup \Sing(\pi) =\emptyset,$ which implies \eqref{wick:prawieglowna}. From our definition it follows that $\Cum_{\pi_f}=\Cumm$ if $\Sing(\pi)=\emptyset$, so by taking the sum over all $\epsilon$ from equation \eqref{wick:prawieglowna}, we see that
\begin{align*}
\state\Big( \big(\mathcal{B}^{\lambda_n}_{\alpha,q}(x_n)-\lambda_n I\big) \cdots \big(\mathcal{B}^{\lambda_1}_{\alpha,q}(x_1)-\lambda_1 I\big)\Big)=\displaystyle\sum_{\pi_f \in \PB_{\geq 2}(n)} \alpha^{\NB(\pi_f)}q^{\rc(\pi)+2 \InNB(\pi_f)}\Cum_{\pi_f}.
\end{align*}
We also see that for
\begin{align*} 
&\state\Big( \big(\mathcal{B}^{\lambda_n}_{\alpha,q}(x_n)-\lambda_n I+\lambda_n I\big) \cdots \big(\mathcal{B}^{\lambda_1}_{\alpha,q}(x_1)-\lambda_1 I+\lambda_1 I\big)\Big)
\intertext{by eqation \eqref{wick:prawieglowna}, we get}
&=\sum_{\nu \subset [n]} \Bigg[\prod_{i\in\nu} \lambda_i \sum_{\pi_f \in \PB_{\geq 2}([n]\setminus \nu)} \alpha^{\NB(\pi_f)}q^{\rc(\pi)+2 \InNB(\pi_f)}\Cum_{\pi_f}\Bigg],
\intertext{which by induction implies \eqref{wick:glowna}.}
\end{align*} 
\end{proof}

\section{Remarks about $(q,t)$-probability space}
We would like to stress that  this section is only loosely related to the previous part of the article.
\subsection{Blitvi{\'c} model} In this section we will show that very similar and interesting partitions appear in the context of $(q,t)$-probability spaces introduced by Blitvi{\'c} \cite{B12,B14}. First, we remind Blitvi{\'c} construction in the context of right creators. 
\noindent Let $\F$ be the (algebraic) full Fock space over $H$ given in \eqref{fockspace}, with property $ \bar{x}=x$. For $ q \in(-1,1)$ and $|q|< t<1$, we define the type $(q,t)$-symmetrization operator on $H^{\otimes n}$ as
\begin{align*}
\tilde{P}_{q,t}^{(n)}:= t^{n\choose 2} P_{0,q/t}^{(n)} \textrm{, } \qquad \tilde{P}_{q,t}:=\bigoplus_{n=0}^\infty \tilde{P}_{q,t}^{(n)},
\end{align*}
and a corresponding inner product as 
\begin{equation}
\langle x_1 \otimes \cdots \otimes x_m, y_1 \otimes \cdots \otimes y_n\rangle_{\widetilde{q,t}}:=\langle x_1 \otimes \cdots \otimes x_m, \tilde{P}_{q,t}^{(n)}(y_1 \otimes \cdots \otimes y_n)\rangle_{0,0}.
\end{equation}
\begin{Defn} For $ q \in(-1,1)$ and $|q|< t<1$, the algebraic full Fock space $\F$ equipped with the inner product $\langle\cdot,\cdot \rangle_{\widetilde{q,t}}$ is called the \emph{$(q,t)-$Fock space}. Blitvi{\'c} introduced the following annihilation operator:
\begin{align*}
&\A(x)(x_1\otimes \cdots \otimes x_n)= \sum_{k=1}^n t^{k-1}q^{n-k} \langle x, x_k \rangle\, x_1\otimes \cdots \otimes \check{x}_k \otimes \cdots \otimes x_n, &&n\geq 1\\ 
&\A(x)\Omega=0,
\intertext{and the creation operator:}
&\A^\ast(x)(x_1 \otimes \cdots \otimes x_n):=  x_1 \otimes \cdots \otimes x_n \otimes x , &&n\geq 1\\
&\A^\ast(x)\Omega=x,
\end{align*}i.e. $\A^\ast(x)$ is adjoint to $\A(x)$ with respect to the inner product $\langle\cdot,\cdot \rangle_{\widetilde{q,t}}$. The operators $\A^\ast(x)$ and $\A(x)$ are called \emph{$(q,t)$-creation and annihilation operators}. Denote by $\tilde{\state}$ the vacuum vector state $\tilde{\state}(X):=\tilde{\state}_{{q,t}}(X)=\langle\Omega, X\Omega\rangle_{\widetilde{q,t}}$.
\end{Defn}
\begin{remark} We restrict our parameters to $q \in(-1,1)$ and $|q|<t<1$, because then $\tilde{P}_{q,t}$ is strictly positive; see \cite[Lemma 4]{B12}. In the article \cite{B12} the allowed range of parameters is $|q|=t <1$, but then $\tilde{P}_{q,t}$ is positive. From the combinatorial point of view, the main result of this section is that Theorem \ref{thm3}) is also true when $|q|=t <1$.
\end{remark} 
\noindent\textbf{$(t,q)$-gauge operator.} Let $T$ be a bounded and self-adjoint operator on $H$. The corresponding \emph{gauge operator} $\tilde{p}(T)$ is an operator on $\F$, defined by
\begin{align*}
& \pa(T) := t^{n-1} p_0(T) R^{(n)}_{0,q/t}=t^{n-1}p_{0,q/t}(T).
\end{align*}
\noindent An explicit form of the operator $\pa(T)$ on $H^{\otimes n}$ is:
\begin{align*}
\pa(T) (x_1\otimes \cdots \otimes x_n)=\sum_{k=1}^n t^{k-1}q^{n-k} x_1\otimes \cdots \otimes \check{x}_k \otimes \cdots \otimes x_n \otimes T(x_k) .
\end{align*} 
\begin{remark} (1). The above definition of a gauge operator is motivated by simply noticing that an annihilator operator is of the form
$\A(x)= t^{n-1}\r(x) R^{(n)}_{0,q/t}$.

\noindent
(2). It is not difficult to see that when the parameters are restricted to $|q|<t<1$, the operator $\pa(T)$ has properties desired in Propositions \ref{Prop:samosprzezone} and \ref{Ansh+}, i.e. if $T$ is self-adjoint, then $\pa(T)$ is self-adjoint, and if $T$ bounded on $H$, then $\pa(T)$ is bounded on the $(q,t)$-Fock space. Indeed, for $f,g\in H^{\otimes n}$ and by Proposition \ref{Prop:samosprzezone} and \ref{Ansh+}, we have 
\begin{align*} &\langle \pa(T)f, g\rangle_{\widetilde{q,t}}=t^{n-1}t^{n\choose 2} \langle p_{0,q/t}(T) f, g\rangle_{{0,q/t}}=t^{n-1}t^{n\choose 2}\langle f, p_{0,q/t}(T^*)g\rangle_{{0,q/t}}=\langle f, \pa(T^*)g\rangle_{\widetilde{q,t}}
\intertext{and }
&\langle \pa(T)f, \pa(T)f \rangle_{\widetilde{q,t}}=t^{2n-2}t^{n\choose 2} \langle p_{0,q/t}(T) f, p_{0,q/t}(T) f \rangle_{{0,q/t}}\leqslant \big(\max\{1,t/(t-q)\}\big)^2\|T\|^2 \|f\|^2_{0,q/t}.
\end{align*} 
Here we also see that under the assumption $|q|=t <1$ we cannot get the above-mentioned property. This property might also be true but then we need a different argument in order to prove it.
\end{remark}
\noindent Now define the following operators
\begin{equation}
\Y(x):= \A(x) +\A^\ast(x)+\pa(T_x) ,\qquad x \in H_\R, 
\end{equation}
where $T_x$ is a bounded self-adjoint operator on $H$, indexed by $x \in H_\R$. The main theorem of this section is a nice Wick formula, which expresses the joint distribution in the collection of their joint cumulants.
\begin{theorem}\label{thm3} Suppose that $(x_1,\dots,x_n) \in H_\R^n$, 
 then 
\begin{align} 
\tilde{\state}\big( \Y(x_n)\cdots \Y(x_{ 1})\big)=\displaystyle\sum_{\pi\in \P_{\geq 2}(n)} q^{\rc(\pi)}t^{\InNA(\pi)}\prod_{B \in\pi }\Big\langle x_{\max(B)}, \prod_{i\in B, i\neq \min(B),\max(B)} T_{x_{i}}x_{\min(B)}\Big\rangle. \label{RownanieGlownetwierdzenieBlitvic}
\end{align}
\end{theorem}
\begin{remark} (1). The proof is similar in spirit to the proof of Theorem \ref{twr:momentogolne}. This is not entirely obvious but we leave the formal proof to the reader (we just sketch a heuristic proof), because it can be obtained by modifications of Theorem \ref{twr:momentogolne}. 

\noindent(2). Similarly as in Corollary \ref{cor13}, we can state that for $t\to 1$, we obtain the $q$-deformed formula for moments of $q$-random variable from the article \cite{Ans01,Ans04b}, and for $T=\mathbf{0}$ we get the formula for moments of $(q,t)$-Gaussian operator from \cite{B12}. 

\noindent(3). The number of \emph{restricted nestings} is defined as $\InNA(\pi):=\#\{ (V,W)\in \nest(\pi)\}$, i.e. it is a number of covered arcs. For a partition in Figure \ref{RestrictedCrossingsBlitvic}, we see that $\rc{(\pi)}=7$ and $\InNA(\pi)=5$. Partitions with restricted crossings and nestings appear in this context in many combinatorial articles; see \cite{KZ06,ChDRS07,RS10}.

\begin{figure}[ht]

\begin{center}
\includegraphics[width=0.7\textwidth]{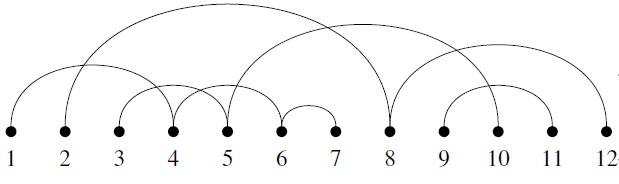}

\caption{A partition of $12$ elements with  four blocks 
}
 \label{RestrictedCrossingsBlitvic} 
\end{center}
\end{figure} 
\end{remark}

\noindent{\emph{The following is the sketch of a proof: }}Let us first observe that it is sufficient to focus on Cases 1, 2a and 3a of Theorem \ref{twr:momentogolne}. Here, we just emphasize how to modify these cases in order to obtain Theorem \ref{thm3}. We keep the notation from this proof, where we assume that all arcs have color $1$, and  the corresponding analog of the statistic $\SLNB$ is redefined as
$$\SLNBlitvic(\pi):=\#\big\{ (V,W)\in \big[Block\primes(\pi)\cup \Sing(\pi)\big]\times \Semi(\pi) \mid \max(V)<i \textrm{ for } i\in W\big\}.$$
In the induction step we assume that 
\begin{align*}
\A^{\epsilon(k)}(x_{k})\cdots \A^{\epsilon(1)}(x_1)\Omega = \sum_{\pi\in\PB_{E;\epsilon}(k)} q^{\rc(\pi) +\InS(\pi)}t^{ \InNA(\pi)+\SLNBlitvic(\pi) } {\mathrm{R}}^\mb{x}_{\pi}\widehat{\mathrm{R}}^\mb{x}_{\pi},
\end{align*}
and analogical map $\pi\to \tilde{\pi}$ from the proof of Theorem \ref{twr:momentogolne} is obtained in the following way:
\begin{enumerate}
\item in Case 1 the operator $ \A^\ast(x_{k+1})$ corresponds to adding the singleton $\{k+1\}$ to $\Sing(\tilde{\pi})$;
\item in Case 2a the annihilator $\A(x_{k+1})$ contributes to a new block in $Block(\tilde{\pi})$, with the last arc $\{\ith,k+1\}$ and $\rc(\tilde{\pi})=\rc(\pi)+u$, $\InS(\tilde{\pi})=\InS(\pi)-u+r$, $\InNA(\tilde{\pi})=\InNA(\pi)+s+t$ and $\SLNBlitvic(\tilde{\pi})=\SLNBlitvic(\pi)+p-(t+s)$. So the exponent of $q$ increases by $r$ and the exponent of $t$ increases by $p$;
\item in Case 3a the gauge operator $\pa(T_{x_{k+1}})$ contributes to a new block in $Block\primes(\tilde{\pi})$ and the change in the statistic is the same as in case (b). 
\end{enumerate} 
Finally, when we put $Block\primes({\pi})\cup \Sing(\pi)=\emptyset$, we obtain the formula \eqref{RownanieGlownetwierdzenieBlitvic}.
\begin{flushright}$\qed$ \end{flushright}
\noindent \textbf{\emph{The orthogonal polynomial.}}
\label{Subsec:Poisson}
\begin{Defn}$(q,t)$-Poisson polynomials are defined by the recursion relations:
\begin{align} \label{wielomianytfreepoisson}
y\tilde{P}_{ n}^{(q,t)}( y) &= \tilde{P}_{ n+1}^{(q,t)}(y) + [n]_{q,t} \tilde{P}_{ n}^{(q,t)}( y) + [n]_{q,t} \tilde{P}_{n-1}^{(q,t)}( y), \qquad n\geq 1
\end{align}
where $[n]_{q,t}=\sum_{i=1}^nq^{i-1}t^{n-i},$ and initial conditions $\tilde{P}_{-1}^{(q,t)}(y) = 0$, $\tilde{P}_{0}^{(q,t)}(y) = 1$ and $\tilde{P}_{1}^{(q,t)}( y) = y$. Let $\tilde{\mu}_{q,t}$ be a probability measure, which associates the orthogonal polynomials $\tilde{P}_{ n}^{(q,t)}$. 
\end{Defn}
\begin{Prop} Suppose that $x \in H, \|x\|=1$ and $T=Id$. Then the probability distribution of $\Y $ with respect to the vacuum state is given by $\tilde{\mu}_{q,t}$.
\end{Prop}
\begin{proof} Note that for $n=1$ $\tilde{P}_{1}^{(q,t)}(\Y(x))\Omega =Y(x)\Omega=x$ and by induction
\begin{align*}
&\tilde{P}_{ n+1}^{(q,t)}( \Y(x))\Omega =\Y(x) \tilde{P}_{ n}^{(q,t)}( \Y(x))\Omega-[n]_{q,t} \tilde{P}_{ n}^{(q,t)}( \Y(x))\Omega - [n]_{q,t} \tilde{P}_{n-1}^{(q,t)}( \Y(x))\Omega\\&=\Y(x) x^{\otimes n}-[n]_{q,t}x^{\otimes n}- [n]_{q,t}x^{\otimes (n-1)}=x^{\otimes n+1}.
\end{align*}
The rest of argument is similar to that of the Proposition \ref{WielomianyTypeB}.
\end{proof}
\begin{Cor} By \cite[Propositions 7A and 7B]{F80} or \cite[Proposition 4.1]{KZ06} we conclude that the moment generating function of the measure $\tilde{\mu}_{q,t}$ has the following elegant continued fraction expansion:
\begin{align} 
\sum_{n\geq 0}m_n(\tilde{\mu}_{q,t})z^n=\cfrac{1}{1-z -\cfrac{z^2}{1-[1]_{q,t}z-\cfrac{[2]_{q,t}z^2}{1-[2]_{q,t}z-\cfrac{[3]_{q,t}z^2}{{1-[3]_{q,t}z-\cfrac{[4]_{q,t}z^2}{\ddots}}}}}}. \label{cotinuedfraction}
\end{align} 
\end{Cor}

\subsection{$t$-deformed free probability}When $q=0$ and $t\in(0,1)$, the case is reduced to a new $t$-deformed free probability (the term ‘$t$-free probability’ also appeared in \cite{BW01}, but in a completely different context of $t$-transformation of measures). Blitvi{\'c} \cite{B12} showed that the statistics of the $t$-deformed semicircular element can be described in an elegant form drawing from the deformed Catalan numbers, the generalized Rogers-Ramanujan continued fraction, and the $t$-Airy function of Ismail \cite{I05}. 
In light of its present interpretation, a $t$-deformed free probability resembles the combinatorial theory of free probability of Voiculescu \cite{V85}. In a $t$-deformed free probability, when we illustrate noncrossing partitions graphically, we connect all consecutive points in a block by a semicircle, and count how many arcs are covered by other arcs; see Figure \ref{LiczenieMomentowTfree}. In particular, we obtain Voiculescu  probability when we pass with $t$ to $1$. This situation was described for pair partitions in \cite{B12}, where it was shown that moments of $t$-semicircular element can be represented by noncrossing nesting pair partitions.

\begin{figure}[h]

\begin{center}
\includegraphics[width=0.7\textwidth]{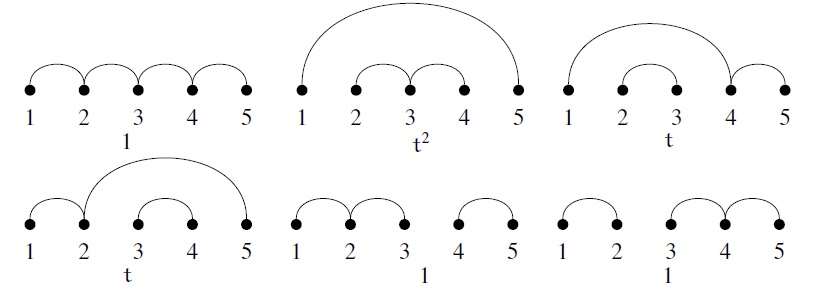}

\caption{Noncrosing cumulants and partitions  for  $\tilde{\state}_{{0,t}}\big(\mathrm{Y}_{0,t}^5(x)\big)=t^2+2t+3$,  with $T=Id$ and $\|x\|=1$. }
 \label{LiczenieMomentowTfree}
\end{center}
\end{figure}

\noindent \textbf{\emph{$t$-free Poisson distribution.}} Al-Salam and Ismail in \cite[Equation 5.6]{AI83} introduced the three-term recurrence 
\begin{align}
y U_n(y,a,b) = U_{n+1}(y,a,b) -at^n U_n(y,a,b) + bt^{n-1} U_{n-1}(y,a,b), \qquad n\geq 1 \label{IsmailSalam1}
\end{align}
with $U_{0}(y)=1$ and $U_{1}(y)=cy$. They also showed \cite[Theorem 5.1]{AI83}, that when $c>0$, $b>0$, $1+at>0$ and $t\in(0,1)$, then there exists a unique purely discrete positive measure with bounded support orthogonalizing polynomial \eqref{IsmailSalam1}.
\begin{Defn}
 For $t\in(0,1)$, we call the measure $\tilde{\mu}_{0,t}$ the $t$-free Poisson or $t$-Marchenko-Pastur distribution.
\end{Defn}

\begin{Prop} \label{DowodDiscritePoisson}
The measure $\tilde{\mu}_{0,t}$ is a purely discrete positive measure with bounded support.
\end{Prop} 
\begin{proof}Recall that random variables $\mathrm{Y}_{0,t}(x)$, with $\|x\|=1$ have the distribution $\tilde{\mu}_{0,t}$, i.e. they ortogonalize \eqref{wielomianytfreepoisson}, with $q=0$. If we put $a=-1$, $b=t^2$ and $c=1$ in the recurrence \eqref{IsmailSalam1} then $1+at=1-t>0$ for $t\in(0,1)$ and
\begin{align}
y U_n(y,-1,t^2) = U_{n+1}(y,-1,t^2) +t^n U_n(y,-1,t^2) + t^{n+1} U_{n-1}(y,-1,t^2), \qquad n\geq 1. \label{pomtfreepoisson}
\end{align}
Now, let us substitute $L_n(y)=U_n(yt,-1,t^2)/t^n$, multiply \eqref{pomtfreepoisson} by $t^{-n-1}$ and replace $y $ by $ty$,  then we get the recursion
$$
y L_n(y) = L_{n+1}(y) +t^{n-1} L_n(y) + t^{n-1} L_{n-1}(y), \qquad n\geq 1 \label{IsmailSalam}
$$
with $L_{0}(y)=1$ and $L_{1}(y)=y$, so we see that $L_n(y)=\tilde{P}_{0}^{(0,t)}(y).$ This observation means that the 
 recurrence \eqref{wielomianytfreepoisson}, with $q=0$ corresponds to monic orthogonal polynomials which orthogonalize  the distribution of $t\mathrm{Y}_{0,t}(x)$. So we deduce from \cite[Theorem 5.1]{AI83} that the law of $t\mathrm{Y}_{0,t}(x)$ is a purely discrete positive measure with bounded support, hence the distribution of $\mathrm{Y}_{0,t}(x) $ (= $\tilde{\mu}_{0,t}$) also has these properties. 
\end{proof}
\begin{remark}(1). In the contexts of combinatorics and of number theory, we can conclude that the $t$-free Poisson distribution is an object of significant interest. It turns out that certain familiar objects, for example the generating function of this distribution, can be represented by the continued fraction expansion \eqref{cotinuedfraction} with $q=0$, which is called the generalized Rogers-Ramanujan continued fraction; see \cite{AI83}.

\noindent (2). In this paper we have shown that the $t$-free Poisson is a discrete probability measure. A similar result was obtained by Blitvi{\'c} \cite{B12} in the case of the $t$-semicircular distribution. This observation suggests, that $t$-free probability is related to paper \cite{CHS15}. 

\noindent (3). The Proposition \ref{DowodDiscritePoisson} can be generalized to the case of $\|x\|\neq 1 $  (with the same proof). Then it can be formulated as follows: if $t\in (0,\min\{1/\|x\|^2,1\})$, then  the probability distribution of $\mathrm{Y}_{0,t}(x)$ is a purely discrete positive measure with bounded support. 
\end{remark}

\begin{center} Acknowledgments
\end{center}
The author would like to thank Marek Bo\.zejko and Franz Lehner for suggesting topics, several discussions and helpful comments. The author also thanks Mourad E. H. Ismail, who suggested the proof of Proposition \ref{DowodDiscritePoisson}. The work was supported by the Austrian Federal Ministry of Education, Science and
  Research and the Polish Ministry of Science and Higher Education, grants
  N$^{\textrm{os}}$  PL 06/2018 and  by the Narodowe Centrum Nauki grant no. 2018/29/B/HS4/01420.

\providecommand{\bysame}{\leavevmode\hbox to3em{\hrulefill}\thinspace}


\begin{thebibliography}{BNTr00}
\bibitem[A01]{Ans01}{Anshelevich, Michael}, \emph{Partition-dependent stochastic measures and {$q$}-deformed cumulants}, {Doc. Math.} {6} (2001), 343--384.
\bibitem[A04a]{Ans04}{Anshelevich, Michael}, \emph{$q$-L\'evy processes}, {J. Reine Angew. Math.} {256} (2004), 181--207.
\bibitem[A04b]{Ans04b}{Anshelevich, Michael}, \emph{Appell polynomials and their relatives}, {Int. Math. Res. Not.} {65} (2004), 3469--3531.
\bibitem[A05]{Ans05}{Anshelevich, Michael}, \emph{Linearization coefficients for orthogonal polynomials using stochastic processes}, {Ann. Probab.} {33} (2005), no.~1, 114--136.
\bibitem[AM12]{AnsMlot12}{Anshelevich, Michael and M{\l}otkowski, Wojciech}, \emph{Semigroups of distributions with linear {J}acobi parameters}, {J. Theoret. Probab.} {25} (2012), no.~4, 1173--1206.
\bibitem[AI83]{AI83}
{Al-Salam, Waleed A. and Ismail, Mourad E. H.}, \emph{Orthogonal polynomials associated with the {R}ogers-{R}amanujan continued fraction}, {Pacific J. Math.} {104} (1983), no.~2, 269--283.
\bibitem[B97]{Bia97}{Biane, Philippe}, \emph{Some properties of crossings and partitions}, {Discrete Math.} {175} (1997), no.~1-3, 41--53.
\bibitem[B12]{B12}{Blitvi{\'c}, Natasha}, \emph{The $(q,t)$-Gaussian process}, {J. Funct. Anal.} {263} (2012), no.~10, 3270--3305.
\bibitem[B14]{B14}{Blitvi{\'c}, Natasha}, \emph{Two-parameter non-commutative central limit theorem}, {Ann. Inst. Henri Poincar\'e Probab. Stat.} {50} (2014), no.~4, 1456--1473.
\bibitem[BKS97]{BKSQGauss}{Bo{\.z}ejko, M., K{\"u}mmerer, B. and Speicher, R.}, \emph{{$q$}-{G}aussian processes: non-commutative and classical aspects}, {Comm. Math. Phys.} {185} (1997), no.~1, 129--154.
\bibitem[BS91]{BS91}{Bo{\.z}ejko, Marek and Speicher, Roland}, \emph{An example of a generalized {B}rownian motion}, {Comm. Math. Phys.} {137} (1991), no.~3, 519--531.
\bibitem[BS94]{BozSpeCoxeter}{Bo{\.z}ejko, Marek and Speicher, Roland}, \emph{Completely positive maps on {C}oxeter groups, deformed commutation relations, and operator spaces}, {Math. Ann.} {300} (1994), no.~1, 97--120.
\bibitem[BW01]{BW01}{Bo{\.z}ejko, Marek and Wysocza{\'n}ski, Janusz}, \emph{Remarks on {$t$}-transformations of measures and convolutions}, {Ann. Inst. H. Poincar\'e Probab. Statist.} {37} (2001), no.~6, 737--761.
\bibitem[BEH15]{BEH15} {Bo{\.z}ejko, M., Ejsmont, W. and Hasebe, T.}, \emph{Fock space associated to {C}oxeter groups of type {B}}, {J. Funct. Anal.} {269} (2015), no.~6, 1769--1795.
\bibitem[ChV06]{ChV06}{Chapoton, F. and Vallette, B.}, \emph{Pointed and multi-pointed partitions of type A and B}, {J. Algebraic Combin.} {23} (2006), no.~4, 295--316.
\bibitem[ChDRS07]{ChDRS07}{Chen, W.Y.C., Deng, E.Y.P., Du, R.R.X., Stanley, R.P. and Yan, C.H.}, \emph{Crossings and nestings of matchings and partitions}, {Trans. Amer. Math. Soc.} {359} (2007), no.~4, 1555--1575. 
\bibitem[CHS15]{CHS15}{Collins, B. Hasebe, T. and Sakuma, N.}, \emph{Free probability for purely discrete eigenvalues of random matrices}, arXiv preprint arXiv:1512.08975 (2015).
\bibitem[DGIX09]{DGIX09}{Delgado, A.M., Geronimo, J.S., Iliev, P. and Xu, Y.}, \emph{On a two-variable class of {B}ernstein-{S}zeg$\ddot{o}$ measures}, {Constr. Approx.} {30} (2009), no.~1, 71--91.
\bibitem[F80]{F80}{Flajolet, P.}, \emph{Combinatorial aspects of continued fractions}, {Discrete Math.} {32} (1980), no.~2, 125--161. 
\bibitem[I05]{I05}{Ismail, Mourad E. H}, \emph{Asymptotics of {$q$}-orthogonal polynomials and a {$q$}-{A}iry function}, {Int. Math. Res. Not.} (2005), no.~18, 1063--1088.
\bibitem[IK12]{IsmaKo12}{Ismail, Mourad E. H. and Koelink, Erik}, \emph{Spectral analysis of certain {S}chr\"odinger operators}, {SIGMA Symmetry Integrability Geom. Methods Appl.} {8} (2012), Paper 061, 19.
\bibitem[L05]{L05}{Lehner, Franz}, \emph{Cumulants in noncommutative probability theory. {III}. {C}reation and annihilation operators on {F}ock spaces}, {Infin. Dimens. Anal. Quantum Probab. Relat. Top.} {8} (2005), no.~3, 407--437.
\bibitem[HP84]{HudPar}{Hudson, R. L. and Parthasarathy, K. R.}, \emph{Quantum {I}to's formula and stochastic evolutions}, {Comm. Math. Phys.} {93} (1984), no.~3, 301--323.
\bibitem[KZ06]{KZ06}{Kasraoui, Anisse and Zeng, Jiang}, \emph{Distribution of crossings, nestings and alignments of two edges in matchings and partitions}, {Electron. J. Combin.} {13} (2006), no.~1, Research Paper 33, 12 pp. (electronic).
\bibitem[M93]{Mol}{M{\o}ller, Jacob Schach}, \emph{Second quantization in a quon-algebra}, {J. Phys. A} {26} (1993), no.~18, 4643--4652.
\bibitem[Nic95]{Nic95}{Nica, Alexandru}, \emph{A one-parameter family of transforms, linearizing convolution laws for probability distributions}, {Comm. Math. Phys.} {168} (1995), no.~1, 187--207.
\bibitem[Nic96]{Nic96}{Nica, Alexandru}, \emph{Crossings and embracings of set-partitions and {$q$}-analogues of the logarithm of the {F}ourier transform}, {Discrete Math.} {157} (1996), no.~1-3 285--309.
\bibitem[N59]{Nel59}{Nelson, Edward}, \emph{Analytic vectors}, {Ann. of Math. (2)} {70} (1959), 572--615.
\bibitem[NS06]{NS06} {Nica, Alexandru and Speicher, Roland}, \emph{Lectures on the combinatorics of free probability}, Cambridge University Press, Cambridge (2006).
\bibitem[R97]{R97}{Reiner, Victor}, \emph{Noncrossing partitions for classical reflection groups}, Discrete Math.\ 177 (1997), 195--222. 
\bibitem[RS10]{RS10}{Rubey, Martin and Stump, Christian}, \emph{Crossings and nestings in set partitions of classical types}, Electron. J. Combin.\ 17 (2010) no.~1, 1077--8926.
\bibitem[RS80]{ReeSim1}{Reed, Michael and Simon, Barry}, \emph{Methods of modern mathematical physics. {I}}, {Academic Press, New York} (1980). 
\bibitem[Sch91]{SchurCondPos}{Sch{\"u}rmann, Michael}, \emph{Quantum stochastic processes with independent additive increments}, {J. Multivariate Anal.} {38} (1991), no.~1, 15--35.
\bibitem[S00]{Sim00}
{Simion, Rodica}, \emph{Combinatorial statistics on type-{B} analogues of noncrossing partitions and restricted permutations}, {Electron. J. Combin.} {7} (2000), no.~9, Research Paper 9, 27 pp. (electronic).
\bibitem[{\'S}00]{Sni}{{\'S}niady, Piotr}, \emph{On $q$-deformed quantum stochastic calculus}, {Probability and Mathematical Statistics} {2} (2001), no.~1, 231--251.
\bibitem[SY00a]{SaiKraw}{Saitoh, Naoko and Yoshida, Hiroaki}, \emph{A $q$-deformed {P}oisson distribution based on orthogonal polynomials}, {J. Phys. A} {33} (2000), no.~7, 1435--1444.
\bibitem[SY00b]{SaiPoisson}{Saitoh, Naoko and Yoshida, Hiroaki}, \emph{$q$-deformed {P}oisson random variables on $q$-{F}ock space}, {J. Math. Phys.} {41} (2000), no.~8, 5767--5772.
\bibitem[V85]{V85} {Voiculescu, Dan}, \emph{Symmetries of some reduced free product $C^\ast$ algebras}, Operator algebras and their connections with topology and ergodic theory, Lect.\ Notes in Math.\ {1132}, Springer, Berlin (1985), 556--588.
\bibitem[V86]{V1} {Voiculescu, Dan}, \emph{Addition of certain noncommuting random variables}, J.\ Funct.\ Anal.\ 66 (1986), no.~3, 323--346.
\bibitem[V91]{V91} {Voiculescu, Dan}, \emph{Limit laws for random matrices and free products}, Invent. Math {104} (1991), no.~1, 201--220.
\bibitem[V14]{V14} {Voiculescu, Dan}, \emph{Free probability for pairs of faces I}, Comm.\ Math.\ Phys. 332 (2014), no.~3, 955--980.
\bibitem[V16]{V16} {Voiculescu, Dan}, \emph{Free probability for pairs of faces {II}: 2-variables bi-free partial {$R$}-transform and systems with rank {$\le 1$} commutation}, Ann. Inst. Henri Poincar\'e Probab. Stat. {52} (2016), no.~1, 1--15.
\end{thebibliography}
\end{document}